\documentclass[a4paper,11pt]{article}
\usepackage{amsmath,amsthm,amssymb,enumitem,xcolor}
\usepackage{tikz}

\usepackage[hang]{footmisc}
\usepackage[nosort,nocompress,noadjust]{cite}

\usepackage[bookmarks=false,hyperfootnotes=false,colorlinks,
    linkcolor={red!60!black},
    citecolor={blue!50!black},
    urlcolor={blue!80!black}]{hyperref}

\renewcommand{\eqref}[1]{\hyperref[#1]{(\ref{#1})}}

\newlist{enumlist}{enumerate}{1}
\setlist[enumlist]{labelindent=0cm,label=\arabic*.,ref=\arabic*,labelwidth=2.5ex,labelsep=0.5ex,leftmargin=3ex,align=left,topsep=0.5ex,itemsep=1ex,parsep=1ex}

\newlist{itemlist}{itemize}{1}
\setlist[itemlist]{labelindent=0cm,label=$\bullet$,labelwidth=2.5ex,labelsep=0.5ex,leftmargin=3ex,align=left,topsep=0.5ex,itemsep=1ex,parsep=1ex}

\numberwithin{equation}{section}

{\theoremstyle{definition}\newtheorem{definition}{Definition}[section]
\newtheorem*{definition*}{Definition}

\newtheorem{remark}[definition]{Remark}

\newtheorem*{example*}{Example}
\newtheorem*{examples*}{Examples}}

\newtheorem{proposition}[definition]{Proposition}
\newtheorem{lemma}[definition]{Lemma}
\newtheorem{theorem}[definition]{Theorem}
\newtheorem{corollary}[definition]{Corollary}

\newtheorem{letterthm}{Theorem}

{\theoremstyle{definition}}

\newcommand{\C}{\mathbb{C}}
\newcommand{\cC}{\mathcal{C}}
\newcommand{\eps}{\varepsilon}
\newcommand{\al}{\alpha}
\newcommand{\be}{\beta}
\newcommand{\ot}{\otimes}
\newcommand{\recht}{\rightarrow}
\newcommand{\Z}{\mathbb{Z}}
\newcommand{\vphi}{\varphi}

\newcommand{\id}{\mathord{\text{\rm id}}}

\newcommand{\N}{\mathbb{N}}

\newcommand{\Om}{\Omega}
\newcommand{\cD}{\mathcal{D}}
\newcommand{\si}{\sigma}
\newcommand{\R}{\mathbb{R}}

\newcommand{\cG}{\mathcal{G}}

\newcommand{\cF}{\mathcal{F}}

\newcommand{\actson}{\curvearrowright}
\newcommand{\cA}{\mathcal{A}}
\newcommand{\cB}{\mathcal{B}}

\newcommand{\cU}{\mathcal{U}}

\newcommand{\cR}{\mathcal{R}}
\newcommand{\dis}{\displaystyle}

\newcommand{\cV}{\mathcal{V}}
\newcommand{\cE}{\mathcal{E}}

\newcommand{\Aut}{\operatorname{Aut}}

\newcommand{\cS}{\mathcal{S}}

\newcommand{\cL}{\mathcal{L}}
\newcommand{\mutil}{\widetilde{\mu}}
\newcommand{\Zai}{Z^1_{\text{\rm ai}}}
\newcommand{\HNN}{\operatorname{HNN}}
\newcommand{\supp}{\operatorname{supp}}
\newcommand{\per}{\operatorname{per}}

\begin{document}

\begin{center}
{\boldmath\LARGE\bf Ergodicity and type of nonsingular Bernoulli actions}

\bigskip

{\sc by Michael Bj\"{o}rklund\footnote{Chalmers, Department of Mathematics, Gothenburg (Sweden).\\ E-mail: micbjo@chalmers.se. M.B. was supported by GoCas Young Excellence grant 11423310}, Zemer Kosloff\footnote{Hebrew University of Jerusalem, Einstein Institute of Mathematics, Jerusalem (Israel).\\ E-mail: zemer.kosloff@mail.huji.ac.il. The research of Z.K. was partially supported by ISF grant No. 1570/17.} and Stefaan Vaes\footnote{\noindent KU~Leuven, Department of Mathematics, Leuven (Belgium).\\ E-mail: stefaan.vaes@kuleuven.be. S.V. is supported by European Research Council Consolidator Grant 614195 RIGIDITY, and by long term structural funding~-- Methusalem grant of the Flemish Government.}}
\end{center}

\begin{abstract}\noindent
\noindent We determine the Krieger type of nonsingular Bernoulli actions $G \actson \prod_{g \in G} (\{0,1\},\mu_g)$. When $G$ is abelian, we do this for arbitrary marginal measures $\mu_g$. We prove in particular that the action is never of type II$_\infty$ if $G$ is abelian and not locally finite, answering Krengel's question for $G = \Z$. When $G$ is locally finite, we prove that type II$_\infty$ does arise. For arbitrary countable groups, we assume that the marginal measures stay away from $0$ and $1$. When $G$ has only one end, we prove that the Krieger type is always I, II$_1$ or III$_1$. When $G$ has more than one end, we show that other types always arise. Finally, we solve the conjecture of \cite{VW17} by proving that a group $G$ admits a Bernoulli action of type III$_1$ if and only if $G$ has nontrivial first $L^2$-cohomology.
\end{abstract}

\section{Introduction}

The Bernoulli actions $G \actson (X,\mu) = (\{0,1\},\mu_0)^G$ of a countable group $G$, given by $(g \cdot x)_h = x_{g^{-1} h}$, play a key role in ergodic theory, measurable group theory and operator algebras. By construction, $\mu$ is a $G$-invariant probability measure. Replacing $\mu$ by an arbitrary product measure $\mu = \prod_{g \in G} \mu_g$, one obtains a very natural family of non measure preserving $G$-actions. Although this construction is straightforward, it turned out to be a very difficult problem to decide when $G \actson (X,\mu)$ is ergodic and, in that case, to determine the Krieger type of the action.

The first results in this direction were providing examples for the group $G = \Z$, through inductive constructions of probability measures $(\mu_n)_{n \in \Z}$ on $\{0,1\}$. It was thus proven in \cite{Kre70} that there exists an ergodic Bernoulli shift without equivalent invariant probability measure, while in \cite{Ham81}, it was shown that there are ergodic Bernoulli shifts of type III, i.e.\ without equivalent $\sigma$-finite invariant measure. Finally, the first examples of Bernoulli shifts of type III$_1$ were constructed in \cite{Kos09}.

Proving ergodicity and determining the type of a nonsingular Bernoulli action is a difficult problem, because these actions may very well be dissipative (i.e.\ admit a fundamental domain), which was already proven in \cite{Ham81}. A first general result was obtained in \cite{Kos12} for $G = \Z$ and marginal measures $(\mu_n)_{n \in \Z}$ satisfying $\mu_n(0) = 1/2$ for all $n \leq 0$~: if such a Bernoulli action is nonsingular and conservative, then it must be either of type II$_1$ or of type III$_1$. In \cite{DL16}, the same result was proven if $\mu_n(0) = p$ for some $p \in (0,1)$ and all $n \leq 0$.

Only very recently, in \cite{VW17}, the first results were established for nonamenable groups $G$. It was conjectured in \cite{VW17} that a countable group $G$ admits a Bernoulli action of type III if and only if the first $L^2$-cohomology $H^1(G,\ell^2(G))$ is nonzero, which is equivalent to saying that $G$ is either infinite amenable or has positive first $L^2$-Betti number. The connection with $L^2$-cohomology stems from the following observation: if the marginal measures $\mu_g$ satisfy $\mu_g(0) \in [\delta,1-\delta]$ for all $g \in G$ and some $\delta > 0$, then by Kakutani's criterion (see \cite{Kak48}), the corresponding Bernoulli action is nonsingular if and only $c_g(h) = \mu_h(0) - \mu_{g^{-1}h}(0)$ has the property that $c_g \in \ell^2(G)$ for all $g \in G$, thus defining a $1$-cocycle $c \in Z^1(G,\ell^2(G))$.

Several families of examples of type III$_1$ Bernoulli actions were obtained in \cite{VW17}, confirming the conjecture for `most' countable groups. Although a general dissipative-conservative criterion was established in \cite{VW17}, it remained an open problem to prove ergodicity and determine the type in full generality.

When $\mu_g(0) \in [\delta,1-\delta]$ for all $g \in G$, it was proven that a conservative Bernoulli action must be ergodic, first for $G = \Z$ and other (amenable) groups satisfying the Hurewicz ratio ergodic theorem in \cite{Kos18}, and then for arbitrary amenable groups in \cite{D18} by using a new pointwise ergodic theorem. If moreover $\lim_{g \recht \infty} \mu_g(0)$ exists, the type of such a Bernoulli action could be determined in \cite{BK18} and is either II$_1$ or III$_1$. As a corollary, it was proven in \cite{BK18} that all infinite amenable groups admit type III$_1$ Bernoulli actions on $\{0,1\}^G$.

In this paper, we completely solve the ergodicity and type problem for arbitrary Bernoulli actions of abelian groups $G$, i.e.\ without assuming that $\mu_g(0)$ stays away from $0$ and $1$. For general countable groups $G$, we prove ergodicity and determine the type for all sufficiently conservative Bernoulli actions with $\mu_g(0)$ staying away from $0$ and $1$. In particular, we solve Krengel's question (see \cite{Kre70} and his MathSciNet review of \cite{Ham81}) and prove that no Bernoulli action of $\Z$ can be of type II$_\infty$. The same holds if $G$ is infinite abelian and not locally finite. However, for infinite locally finite groups, we construct Bernoulli actions of type II$_\infty$. We also confirm the conjecture of \cite{VW17} and prove that all groups $G$ with nontrivial $H^1(G,\ell^2(G))$ admit Bernoulli actions of type III$_1$.

The main results in this paper are the following. The first result deals with ergodicity. Recall that a nonsingular action $G \actson (X,\mu)$ is ergodic if every $G$-invariant Borel set $\cU \subset X$ has measure zero or measure one. The action is weakly mixing if the diagonal action $G \actson (X \times Y,\mu \times \eta)$ stays ergodic for every ergodic probability measure preserving action $G \actson (Y,\eta)$.

\begin{letterthm}[See Theorems \ref{thm.dichotomy-dissipative-weakly-mixing} and \ref{thm.general-ergodicity}]\label{thm-A}
If $G$ is an infinite abelian group, any nonsingular Bernoulli action $G \actson (X,\mu) = \prod_{g \in G} (\{0,1\},\mu_g)$ with $\mu$ nonatomic is either dissipative or weakly mixing.

If $G$ is any infinite group and $G \actson (X,\mu) = \prod_{g \in G} (\{0,1\},\mu_g)$ is a nonsingular Bernoulli action that is strongly conservative and satisfies $\mu_g(0) \in [\delta,1-\delta]$ for all $g \in G$, then $G \actson (X,\mu)$ is weakly mixing.
\end{letterthm}

Strong conservativeness is introduced in Definition \ref{def.strong-conservative}. For amenable groups, it is equivalent to the usual notion of conservativeness (see Proposition \ref{prop.characterize-strong-conservative}). In Proposition \ref{prop.strong-conservative-Bernoulli}, we also provide an easy to check sufficient condition for the strong conservativeness of a nonsingular Bernoulli action~: if $\mu_g(0) \in [\delta,1-\delta]$ for all $g \in G$ and if $\kappa > \delta^{-1} (1-\delta)^{-1}$, it suffices that
\begin{equation}\label{eq.growth-cocycle}
\sum_{g \in G} \exp(-4\kappa \, \|c_g\|_2^2) =+\infty \; , \quad\text{where}\quad c_g(h) = \mu_h(0) - \mu_{g^{-1}h}(0) \; .
\end{equation}
So strong conservativeness (and hence also weak mixing, by Theorem \ref{thm-A}) holds if the function $g \mapsto \|c_g\|_2$ does not grow fast to infinity. Condition \eqref{eq.growth-cocycle} first appeared in \cite[Proposition 4.1]{VW17}, where it was shown that for $\kappa$ large enough, \eqref{eq.growth-cocycle} implies that $G \actson (X,\mu)$ is conservative, while if \eqref{eq.growth-cocycle} fails for $\kappa = 1/8$, then $G \actson (X,\mu)$ is dissipative.

We determine the Krieger type of the nonsingular Bernoulli actions appearing in Theorem \ref{thm-A}. We deal separately with abelian groups and arbitrary countable groups.

\begin{letterthm}[see Theorem \ref{thm.type-abelian}]\label{thm-B}
Let $G$ be an infinite abelian, non locally finite group and $G \actson (X,\mu) = \prod_{g \in G} (\{0,1\},\mu_g)$ any nonsingular Bernoulli action. Assume that $\mu$ is nonatomic and that $G \actson (X,\mu)$ is not dissipative. Then, $G \actson (X,\mu)$ is weakly mixing and its type is given as follows.
\begin{enumlist}
\item If $\lambda \in (0,1)$ and $\sum_{g \in G} (\mu_g(0)-\lambda)^2 < +\infty$, then $G \actson (X,\mu)$ is of type II$_1$.
\item If $\lim_{g \recht \infty} \mu_g(0) = \lambda \in (0,1)$ and $\sum_{g \in G} (\mu_g(0)-\lambda)^2 = +\infty$, then $G \actson (X,\mu)$ is of type III$_1$.
\item If $\lim_{g \recht \infty} \mu_g(0)$ equals $0$ or $1$, then $G \actson (X,\mu)$ is of type III.
\item If $\mu_g(0)$ does not converge as $g \recht \infty$, then $G \actson (X,\mu)$ is of type III$_1$.
\end{enumlist}
\end{letterthm}

Theorem \ref{thm-B} thus provides a complete answer to the question which types can arise for nonsingular Bernoulli actions of the group of integers $\Z$. In particular, the group of integers does not admit Bernoulli actions of type II$_\infty$. Only in the most delicate case where $\lim_{g \recht \infty} \mu_g(0)$ equals $0$ or $1$, it is unclear which of the types III$_\lambda$, $\lambda \in [0,1]$, can be realized.

Surprisingly, infinite locally finite groups admit nonsingular Bernoulli actions of type II$_\infty$, as well as type III$_\lambda$ for any $\lambda \in [0,1]$, see Proposition \ref{prop.locally-finite-type-II-infty}. The reason why locally finite groups $G$ behave differently is because they have infinitely many ends: they admit many subsets $W \subset G$ such that $W$ and the complement $G \setminus W$ are infinite, but $|g W \vartriangle W| < \infty$ for all $g \in G$.

For arbitrary countable groups $G$, we assume that $\mu_g(0) \in [\delta,1-\delta]$ for all $g \in G$ and that \eqref{eq.growth-cocycle} holds for some $\kappa > \delta^{-1} (1-\delta)^{-1}$. Again, we have to distinguish between the case where $G$ has only one end and the case where $G$ has infinitely many ends. Generically, an infinite group has one end. By Stallings theorem (see Remark \ref{rem.ends-groups}), the only other possibilities are zero ends (for finite groups), two ends (for virtually cyclic groups) and infinitely many ends (for locally finite groups, as well as for certain amalgamated free products and HNN extensions).

\begin{letterthm}[See Theorem \ref{thm.type-bernoulli}]\label{thm-C}
Let $G$ be a group with one end and let $G \actson (X,\mu) = \prod_{g \in G} (\{0,1\},\mu_g)$ be a nonsingular Bernoulli action with $\mu_g(0) \in [\delta,1-\delta]$ for all $g \in G$. Assume that \eqref{eq.growth-cocycle} holds for a $\kappa > \delta^{-1} (1-\delta)^{-1}$. Then, $G \actson (X,\mu)$ is weakly mixing and of type III$_1$, unless $\sum_{g \in G} (\mu_g(0)-\lambda)^2 < +\infty$ for some $\lambda \in (0,1)$, in which case the action is of type II$_1$.
\end{letterthm}

In Theorem \ref{thm.type-bernoulli}, we also determine the type of $G \actson (X,\mu)$ when $G$ has more than one end. In that case, there always exist Bernoulli actions of type III$_\lambda$ with $\lambda \neq 1$. As a corollary, we prove the following result, confirming the conjecture made in \cite{VW17}.

\begin{letterthm}[see Corollary \ref{cor.solution-VW-conjecture}]\label{thm-D}
A countable infinite group $G$ admits a nonsingular Bernoulli action of type III$_1$ if and only if $H^1(G,\ell^2(G)) \neq \{0\}$.
\end{letterthm}

We use the following approach to prove Theorems \ref{thm-A} to \ref{thm-D}. Let $G \actson (X,\mu) = \prod_{g \in G} (\{0,1\},\mu_g)$ be a nonsingular Bernoulli action. Denoting by $\al : G \times X \recht \R$ the logarithm of the Radon-Nikodym cocycle, we consider the Maharam extension $G \actson X \times \R$ given by
$$g \cdot (x,t) = (g \cdot x, \al(g,x) + t) \; ,$$
which is an infinite measure preserving action. Proving ergodicity and determining the type of $G \actson (X,\mu)$ amounts to determining the $G$-invariant functions $L^\infty(X \times \R)^G$.

Denote by $\cS_G$ the group of finite permutations of the countable set $G$. Consider the natural nonsingular action $\cS_G \actson (X,\mu)$ given by permuting the coordinates of the infinite product $(X,\mu) = \prod_{g \in G} (\{0,1\},\mu_g)$. As in \cite{Kos18,D18,BK18}, the main step is to prove that a $G$-invariant function $F \in L^\infty(X \times \R)$ is automatically invariant under the Maharam extension of $\cS_G \actson (X,\mu)$. In \cite{Kos18,D18,BK18}, essential use is made of pointwise ergodic theorems for nonsingular group actions $G \actson (X,\mu)$, so that only amenable groups can be treated. For general, possibly nonamenable groups $G$, we do not use pointwise ergodic theorems, but exploit a new notion of strong conservativeness that we introduce in Section \ref{sec.strongly-conservative} and that is of independent interest.

Also, all past results determining the type of a nonsingular Bernoulli action $G \actson (X,\mu)$ were making use of rather restrictive assumptions on the measures $\mu_g$, like the existence of a limit for $\mu_g(0)$ when $g$ tends to infinity in specified directions. In this paper, we no longer need this assumption and, roughly speaking, use instead weak limits of probability distributions given by the values $\mu_g(0)$ as $g$ tends to infinity (see the proof of Lemma \ref{lem.general-automatic-invariance-permutation-group}).

In the special case where $G$ is abelian, we no longer need to make the assumption that $\mu_g(0)$ stays away from $0$ and $1$. Theorem \ref{thm-B} is the first result determining the type of a nonsingular Bernoulli action in such a situation. An important step in the proof is to show that even when $\mu_g(0)$ tends to zero, for every $a \in G$ and for `most' $g$ tending to infinity, the quotient $\mu_{ga}(0)/\mu_g(0)$ converges to $1$ (see the proof of Lemma \ref{lem.automatic-invariance-permutation-group}).

Once invariance under $\cS_G$ is proven, we can invoke \cite{AP77,SV77}, where it is shown when $\cS_G \actson (X,\mu)$ is ergodic and what its type is. In Section \ref{sec.permutation-action}, we recall these results on permutation actions.

\section{Preliminaries}

\subsection{Nonsingular group actions: notations and terminology}

Given a standard probability space $(X,\mu)$, a Borel map $\al : X \recht X$ is called \emph{nonsingular} if $\mu(\al^{-1}(\cU)) = 0$ whenever $\cU \subset X$ is a Borel set of measure zero. When $\al$ is a nonsingular Borel bijection, we denote the Radon-Nikodym derivative between $\mu$ and $\mu \circ \al$ as
\begin{align*}
& \frac{d\mu(\al(x))}{d\mu(x)} = \frac{d \mu \circ \al}{d\mu} (x) = \frac{d \al^{-1} \mu}{d\mu}(x) \quad\text{such that}\\ &\int_X F(x) \frac{d\mu(\al(x))}{d\mu(x)} \, d\mu(x) = \int_X F(\al^{-1}(x)) \, d\mu(x) \quad\text{for all $F \in L^\infty(X)$.}
\end{align*}

An action $\al$ of a countable group $G$ by Borel bijections of $X$ is called nonsingular if $\al_g$ is nonsingular for every $g \in G$. Such a nonsingular action $G \actson (X,\mu)$ is called \emph{ergodic} if all $G$-invariant Borel subsets of $X$ have measure $0$ or $1$. The action is called \emph{(essentially) free} if $g \cdot x \neq x$ for all $g \in G \setminus \{e\}$ and a.e.\ $x \in X$.

A free nonsingular action $G \actson (X,\mu)$ is called \emph{conservative} if for every nonnegligible Borel set $\cU \subset X$, there exists a $g \in G \setminus \{e\}$ such that $\mu(\cU \cap g \cdot \cU) > 0$. The action is called \emph{dissipative} if there exists a Borel set $\cU \subset X$ such that all $(g \cdot \cU)_{g \in G}$ are disjoint and $\bigcup_{g \in G} g \cdot \cU$ has measure $1$. Whenever $G \actson (X,\mu)$ is free and nonsingular, there exist disjoint $G$-invariant Borel sets $X_0,X_1 \subset X$ such that $G \actson X_0$ is conservative, $G \actson X_1$ is dissipative and $\mu(X_0 \cup X_1) = 1$. Up to sets of measure zero, $X_0$ and $X_1$ are unique. Also,
$$X_1 = \Bigl\{ x \in X \Bigm| \sum_{g \in G} \frac{d\mu(g \cdot x)}{d\mu(x)} < +\infty\Bigr\} \; .$$
For more details on conservativeness of group actions, see e.g.\ \cite[Chapter 1]{Aar97}.

\subsection{Maharam extension and type}

Fix a countable group $G$ and a nonsingular action $G \actson (X,\mu)$. The map
$$\al : G \times X \recht \R : \al(g,x) = \log \frac{d\mu(g \cdot x)}{d\mu(x)}$$
satisfies the $1$-cocycle identity $\al(gh,x) = \al(g,h \cdot x) + \al(h,x)$ for all $g,h \in G$ and a.e.\ $x \in X$. Then,
$$G \actson X \times \R : g \cdot (x,t) = (g \cdot x,t + \al(g,x))$$
is a group action, called the \emph{Maharam extension} of $G \actson (X,\mu)$ (see \cite{Mah63}). It is easy to check that the (infinite) measure $d \mu \times e^{-t} dt$ on $X \times \R$ is $G$-invariant. The action $G \actson X \times \R$ commutes with the action $\R \actson X \times \R$ given by $s \cdot (x,t) = (x,t+s)$. Defining $(Z,\eta)$ as the space of ergodic components of $G \actson (X,\mu)$, so that the algebra of $G$-invariant functions $L^\infty(X \times \R)^G$ can be identified with $L^\infty(Z,\eta)$, we find the nonsingular action $\R \actson (Z,\eta)$, which is called the \emph{associated flow} of $G \actson (X,\mu)$ (see \cite{Kri75}). Up to conjugacy, this flow only depends on the measure class of $\mu$.

Assume now that $G \actson (X,\mu)$ is ergodic. The \emph{ratio set} $r(\al)$ of $G \actson^\al (X,\mu)$ (see \cite{Kri70}) is the closed subset of $[0,+\infty)$ consisting of all $t \in [0,+\infty)$ satisfying the following property: for every $\eps > 0$ and for every nonnegligible Borel set $\cU \subset X$, there exists a $g \in G$ and a nonnegligible Borel set $\cV \subset \cU$ such that $g \cdot \cV \subset \cU$ and $|d\mu(g \cdot x)/d\mu(x) - t| < \eps$ for every $x \in \cV$. Since $r(\al) \cap (0,+\infty)$ is a closed multiplicative subgroup of $\R^*_+$, we are thus in precisely one of the following cases. We refer to \cite{Kri70} for further details.
\begin{enumlist}
\item $r(\al) = \{1\}$. This corresponds to the case where there exists a $G$-invariant measure $\mu_1 \sim \mu$. When $\mu$ is atomic, the action $G \actson (X,\mu)$ is said to be of type I. When $\mu$ is nonatomic and $\mu_1(X) < +\infty$, the action $G \actson (X,\mu)$ is said to be of type II$_1$. When $\mu$ is nonatomic and $\mu_1(X) = +\infty$, the action $G \actson (X,\mu)$ is said to be of type II$_\infty$.
\item $r(\al) = \{0,1\}$. Then the action $G \actson (X,\mu)$ is said to be of type III$_0$.
\item $r(\al) =\{0\} \cup \{\lambda^n \mid n \in \Z\}$ for some $\lambda \in (0,1)$. Then the action $G \actson (X,\mu)$ is said to be of type III$_\lambda$.
\item $r(\al) = [0,+\infty)$. Then the action $G \actson (X,\mu)$ is said to be of type III$_1$.
\end{enumlist}

The types are also determined by the associated flow of the action, see \cite{Kri75}. The action $G \actson (X,\mu)$ is of type I or II if and only if the associated flow is conjugate with the translation action of $\R$ on itself. The action is of type III$_\lambda$ with $\lambda \in (0,1)$ if and only if the associated flow is conjugate with the translation action of $\R$ on $\R / \Z \log \lambda$. The action is of type III$_1$ if and only if the associated flow is trivial, meaning that the Maharam extension $G \actson X \times \R$ is ergodic. Finally, the action is of type III$_0$ if and only if the associated flow is properly ergodic, meaning that every orbit of $\R \actson (Z,\eta)$ has measure zero. It also follows that $\log(r(\al) \cap (0,+\infty))$ is the kernel of the associated flow, i.e.\ the subgroup of $s \in \R$ that act trivially on $Z$ a.e.

Both the ratio set and the associated flow of an ergodic nonsingular action $G \actson (X,\mu)$ only depend on the \emph{orbit equivalence relation} $\cR(G \actson X) = \{(g \cdot x,x) \mid g \in G, x\in X\}$. The associated flow is identical to the flow of weights of the von Neumann algebra of $\cR(G \actson X)$.

\subsection{Permutation actions on infinite products}\label{sec.permutation-action}

Given a countably infinite set $I$, denote by $\cS_I$ the group of finite permutations of $I$. Whenever $(X,\mu) = \prod_{i \in I} (\{0,1\},\mu_i)$ is a product probability space with $\mu_i(0) \in (0,1)$ for all $i \in I$, we consider the natural nonsingular action $\cS_I \actson (X,\mu)$ given by permuting the coordinates: $(\si \cdot x)_i = x_{\si^{-1}(i)}$ for all $x \in X$, $i \in I$, $\si \in \cS_I$.

Note that $(X,\mu) = \prod_{i \in I} (\{0,1\},\mu_i)$ is nonatomic if and only if
\begin{equation}\label{eq.nonatomic-condition}
\sum_{i \in I} \min \{\mu_i(0),\mu_i(1)\} = +\infty \; .
\end{equation}

By \cite[Theorem 1.1]{AP77}, the action $\cS_I \actson (X,\mu)$ is ergodic if and only if the nonatomicity condition \eqref{eq.nonatomic-condition} holds. Assuming that \eqref{eq.nonatomic-condition} holds, \cite[Theorem 1.2]{SV77} says that $\cS_I \actson (X,\mu)$ is of type III, unless there exists a subset $I_0 \subset I$ with infinite complement $I_1 \subset I$ and a $\lambda \in (0,1)$ such that
$$\sum_{i \in I_0} \min \{\mu_i(0),\mu_i(1)\} + \sum_{i \in I_1} (\mu_i(0) - \lambda)^2 <+\infty \; .$$
In the latter case, the action is of type II$_1$ when $I_0$ is finite and of type II$_\infty$ when $I_0$ is infinite.

The following result is an immediate consequence of \cite[Proposition 1.5]{DL16}. The first point is a special case of \cite[Remark 1.7]{DL16}. For completeness, we give a detailed argument.

\begin{proposition}\label{prop.III-1-permutation-action}
Let $I$ be a countably infinite set and $(X,\mu) = \prod_{i \in I} (\{0,1\},\mu_i)$ a product probability space with $\mu_i(0) \in (0,1)$ for all $i \in I$.
\begin{enumlist}
\item If $\lim_{i \recht \infty} \mu_i(0) = \lambda \in (0,1)$ and $\sum_{i \in I} (\mu_i(0) - \lambda)^2 = +\infty$, then $\cS_I \actson (X,\mu)$ is ergodic and of type III$_1$.
\item If $\lambda,\lambda' \in (0,1)$ are both limit values of $\mu_i(0)$ for $i \recht \infty$, then $\cS_I \actson (X,\mu)$ is ergodic and
$$\frac{\lambda \; (1-\lambda')}{(1-\lambda) \; \lambda'}$$
belongs to the ratio set of $\cS_I \actson (X,\mu)$.
\end{enumlist}
\end{proposition}

In the proof of Proposition \ref{prop.III-1-permutation-action}, we make use of the so called \emph{homoclinic} equivalence relation $\cR_I$ on $(X,\mu)$ defined by $x \sim x'$ if and only if $x_i = x'_i$ for all but finitely many $i \in I$. Note that the orbit equivalence relation of $\cS_I \actson (X,\mu)$ is a subequivalence relation of $\cR_I$. For every $\rho \in \R$, we consider the $1$-cocycle
\begin{equation}\label{eq.modified-1-cocycle}
\gamma_\rho : \cR_I \recht \R : \gamma_\rho(x',x) = \sum_{i \in I} (\log \mu_i(x'_i) - \log \mu_i(x_i)) - \sum_{i \in I} \rho (x'_i - x_i) \; .
\end{equation}

\begin{proof}
If $\mu_i(0)$ admits a limit value in $(0,1)$, it follows that \eqref{eq.nonatomic-condition} holds. So by \cite[Theorem 1.1]{AP77}, the action $\cS_I \actson (X,\mu)$ is ergodic.

We denote by
$$\beta : \cS_I \times X \recht \R : \beta(\si,x) = \sum_{i \in I} (\log \mu_i((\si \cdot x)_i) - \log \mu_i(x_i))$$
the logarithm of the Radon-Nikodym cocycle for $\cS_I \actson (X,\mu)$. Consider the Maharam extension $\cS_I \actson X \times \R$ given by $\si \cdot (x,t) = (\si \cdot x, \be(\si,x) + t)$.

Assume that $\lambda \in (0,1)$ is a limit point of $\mu_i(0)$ as $i \recht \infty$. Write $\rho = \log((1-\lambda)/\lambda)$. Let $F \in L^\infty(X \times \R)$ be an $\cS_I$-invariant function. We prove that $F$ is $\gamma_\rho$-invariant, in the sense that
\begin{equation}\label{eq.gamma-rho-invariance}
F(x',\gamma_\rho(x',x) + t) = F(x,t) \quad\text{for a.e.\ $(x',x,t) \in \cR_I \times \R$.}
\end{equation}
Fix $i \in I$ and take a sequence $j_n \in I \setminus \{i\}$ such that $j_n \recht \infty$ and $\lim_n \mu_{j_n}(0) = \lambda$. For every $i,j \in I$, denote by $\si_{i,j}$ the permutation of $i$ and $j$. Identify $(X,\mu) = (\{0,1\},\mu_i) \times (X_0,\mu_0)$, where $(X_0,\mu_0) = \prod_{j \neq i} (\{0,1\},\mu_j)$, and view elements of $X$ as pairs $(x_i,x)$ with $x_i \in \{0,1\}$ and $x \in X_0$. Since $F$ is $\cS_I$-invariant, we have
\begin{equation}\label{eq.intermediate-expr}
F(\si_{i,j_n} \cdot (x_i,x), t) = F(x_i,x,t-R_{i,j_n}(x_i,x_{j_n})) \quad\text{for all $n$ and a.e.\ $(x_i,x,t) \in X \times \R$,}
\end{equation}
where $R_{i,j}(k,l) = \log(\mu_i(l) \, \mu_j(k)) - \log(\mu_i(k) \, \mu_j(l))$. Denote by $\nu$ the probability measure on $\R$ given by $d\nu(t) = (1/2) e^{-|t|} dt$. Whenever $\cF \subset I \setminus \{i\}$ is a finite subset, $F_0 \in L^\infty(\{0,1\} \times \{0,1\}^\cF \times \R)$ and $H \in L^\infty(\{0,1\}^\cF \times \R)$, we have
\begin{multline*}
\int_{X_0 \times \R} F_0(\si_{i,j_n} \cdot (0,x), t) \, H(x,t) \, d\mu_0(x) \, d\nu(t) \\ = \int_{X_0 \times \R} (\mu_{j_n}(0) F_0(0,x,t) + \mu_{j_n}(1) F_0(1,x,t)) \, H(x,t) \, d\mu_0(x) \, d\nu(t)
\end{multline*}
whenever $j_n \not\in \cF$. Since the Radon-Nikodym derivative of $\si_{i,j_n}$ stays bounded, we can approximate $F$ in $L^1$-norm by such an $F_0$ and conclude that
\begin{multline*}
\lim_{n \recht \infty} \int_{X_0 \times \R} F(\si_{i,j_n} \cdot (0,x), t) \, H(x,t) \, d\mu(x) \, d\nu(t) \\ = \int_{X_0 \times \R} (\lambda \, F(0,x,t) + (1-\lambda)\, F(1,x,t)) \, H(x,t) \, d\mu(x) \, d\nu(t)
\end{multline*}
for all $H \in L^\infty(\{0,1\}^\cF \times \R)$ and all finite subsets $\cF \subset I \setminus \{i\}$. Write $R_i = \log(\mu_i(1) / \mu_i(0))$. Since $R_{i,j_n}(0,0) = 0$ and $R_{i,j_n}(0,1) \recht R_i - \rho$, we similarly find that
\begin{multline*}
\lim_{n \recht \infty} \int_{X_0 \times \R} F(0,x,t-R_{i,j_n}(0,x_{j_n})) \, H(x,t) \, d\mu(x) \, d\nu(t) \\ = \int_{X_0 \times \R} (\lambda \, F(0,x,t) + (1-\lambda) \, F(0,x,t-R_i+\rho)) \, H(x,t) \, d\mu(x) \, d\nu(t) \; .
\end{multline*}
In combination with \eqref{eq.intermediate-expr}, it follows that
$$\lambda \, F(0,x,t) + (1-\lambda) \, F(1,x,t) = \lambda \, F(0,x,t) + (1-\lambda) \, F(0,x,t-R_i+\rho)$$
so that $F(1,x,t) = F(0,x,t-R_i+\rho)$ for a.e.\ $(x,t) \in X_0 \times \R$. Denoting by
$$\tau_i : X \recht X : (\tau_i(x))_j = x_j \;\;\text{if $j \neq i$, and}\;\; (\tau_i(x))_i = 1 - x_i$$
the transformation given by changing the $i$'th coordinate, this means that
$$F(\tau_i(x),\gamma_\rho(\tau_i(x),x) + t) = F(x,t) \quad\text{for a.e. $(x,t) \in X \times \R$.}$$
Since the graphs of the transformations $(\tau_i)_{i \in I}$ generate the equivalence relation $\cR_I$, it follows that \eqref{eq.gamma-rho-invariance} holds.

Under the hypothesis of the first point of the proposition, \cite[Proposition 1.5]{DL16} says that the $1$-cocycle $\gamma_\rho$ is ergodic. This precisely means that $F$ must be essentially constant, so that $\cS_I \actson (X,\mu)$ is of type III$_1$.

When also $\lambda' \in (0,1)$ is a limit value of $\mu_i(0)$ for $i \recht \infty$, we define $\rho' = \log((1-\lambda')/\lambda')$. Then \eqref{eq.gamma-rho-invariance} holds for both $\rho$ and $\rho'$. It follows that $F(x,t) = F(x,t-\rho+\rho')$ for a.e.\ $(x,t) \in X \times \R$. So, $\exp(\rho-\rho')$ belongs to the ratio set of $\cS_I \actson (X,\mu)$.
\end{proof}

\subsection{Nonsingular Bernoulli actions}

Given a countable group $G$ and a family of probability measures $(\mu_g)_{g \in G}$ on $\{0,1\}$ satisfying $\mu_g(0) \in (0,1)$ for all $g \in G$, we consider the Bernoulli action given by
\begin{equation}\label{eq.nonsingular-bernoulli}
G \actson (X,\mu) = \prod_{g \in G} (\{0,1\},\mu_g) : (g \cdot x)_h = x_{g^{-1} h} \; .
\end{equation}
Occasionally, we refer to \eqref{eq.nonsingular-bernoulli} as the left Bernoulli action, while the right Bernoulli action is given by $(g \cdot x)_{h} = x_{hg}$.

By Kakutani's criterion for the equivalence of product measures in \cite{Kak48}, the Bernoulli action $G \actson (X,\mu)$ is nonsingular if and only if
\begin{equation}\label{eq.criterion-nonsingular}
\sum_{h \in G} \bigl(\sqrt{\vphantom{\mu_h}\smash[b]{\mu_{gh}(0)}} - \sqrt{\mu_h(0)}\bigr)^2 + \sum_{h \in G} \bigl(\sqrt{\vphantom{\mu_h}\smash[b]{\mu_{gh}(1)}} - \sqrt{\mu_h(1)}\bigr)^2 < +\infty \quad\text{for all $g \in G$.}
\end{equation}
If there exist a $\delta > 0$ such that $\mu_g(0) \in [\delta,1-\delta]$ for all $g \in G$, then \eqref{eq.criterion-nonsingular} is equivalent with the condition
$$\sum_{h \in G} (\mu_{gh}(0) - \mu_h(0))^2 < +\infty \quad\text{for all $g \in G$.}$$
In that case, $c : G \recht \ell^2(G) : c_g(h) =  \mu_h(0) - \mu_{g^{-1}h}(0)$ is a well defined $1$-cocycle for $G$ with values in the left regular representation.

The product measure $\mu$ in \eqref{eq.nonsingular-bernoulli} is nonatomic if and only if
$$\sum_{g \in G} \min \{\mu_g(0),\mu_g(1)\} =+\infty \; .$$

For completeness, we include a proof of the following elementary lemma.

\begin{lemma}\label{lem.essential-free}
Let $G \actson (X,\mu)$ be a Bernoulli action as in \eqref{eq.nonsingular-bernoulli}. Assume that $\mu$ is nonatomic and that $G \actson (X,\mu)$ is nonsingular. Then $G \actson (X,\mu)$ is essentially free.
\end{lemma}
\begin{proof}
Let $g \neq e$. We have to prove that $\{x \in X \mid g \cdot x = x \}$ has measure zero. First assume that $g$ has infinite order. Choosing representatives for the left cosets of the subgroup $g^\Z \subset G$, we find a subset $H \subset G$ such that $G = \bigsqcup_{n \in \Z} g^n H$. Define $I = \bigcup_{n \; \text{even}} g^n H$ and $J = \bigcup_{n \; \text{odd}} g^n H$. Identify $(X,\mu) = (X_0,\mu_0) \times (X_1,\mu_1)$, where
$$(X_0,\mu_0) = \prod_{h \in I} (\{0,1\},\mu_h) \quad\text{and}\quad (X_1,\mu_1) = \prod_{h \in J} (\{0,1\},\mu_h) \; .$$
We then find nonsingular bijections $\al : X_0 \recht X_1$ and $\be : X_1 \recht X_0$ such that $g \cdot (x_0,x_1) = (\be(x_1),\al(x_0))$. Since $\mu$ is nonatomic, both $\mu_0$ and $\mu_1$ are nonatomic. If $g \cdot (x_0,x_1) = (x_0,x_1)$, we must have that $x_1 = \al(x_0)$ and the set of points $\{(x_0,\al(x_0)) \mid x_0 \in X_0\}$ has measure zero.

If $g$ has finite order $n \geq 2$, we find a subset $H \subset G$ such that $G = \bigsqcup_{k=0}^{n-1} g^k H$. We define $(X_k,\mu_k)$ by taking the product over $g^k H$ and then reason similarly as above.
\end{proof}

\section{Ergodicity and type for Bernoulli actions of abelian groups}

In this section, we prove the first part of Theorem \ref{thm-A} and Theorem \ref{thm-B}. The key technical lemma is the following result, saying that a nonsingular Bernoulli action of an abelian group $G$ is either dissipative, or has the property that any $G$-invariant function for the Maharam extension is automatically invariant under the Maharam extension of the permutation action of $\cS_G$. We actually prove this general dichotomy lemma for arbitrary amenable groups $G$ and Bernoulli actions for which also the right shift is nonsingular.

\begin{lemma}\label{lem.automatic-invariance-permutation-group}
Let $G$ be an amenable group and $G \actson (X,\mu) = \prod_{g \in G} (\{0,1\},\mu_g)$ any nonsingular Bernoulli action, with $\mu_g(0) \in (0,1)$ for all $g \in G$. Assume that also the right Bernoulli shift is nonsingular. Assume that $\mu$ is nonatomic and that $G \actson (X,\mu)$ is not dissipative.

If $G \actson (Y,\eta)$ is any probability measure preserving (pmp) action and $G \actson X \times \R \times Y$ is its diagonal product with the Maharam extension of $G \actson (X,\mu)$, then any $G$-invariant function $F \in L^\infty(X \times \R \times Y)$ satisfies
\begin{equation}\label{eq.extra-invariance}
F(\si \cdot x , t + \be(\si,x), y) = F(x,t,y) \quad\text{for all $\si \in \cS_G$ and a.e.\ $(x,t,y) \in X \times \R \times Y$,}
\end{equation}
where $\be(\si,x) = \log(d\mu(\si \cdot x)/d\mu(x))$.
\end{lemma}
\begin{proof}
For every $a,b \in G$ with $a \neq b$, denote by $\si_{a,b} \in \cS_G$ the permutation of $a$ and $b$. Denote, for $x_0,x_1 \in \{0,1\}$,
$$D_{a,b}(x_0,x_1) = \frac{\mu_a(x_1) \, \mu_b(x_0)}{\mu_a(x_0) \, \mu_b(x_1)} \quad\text{so that}\quad \frac{d\mu(\si_{a,b} \cdot x)}{d\mu(x)} = D_{a,b}(x_a,x_b) \quad\text{for all $x \in X$.}$$
We then get that
\begin{equation}\label{eq.basic-formula-flip}
\frac{d\mu(g \cdot (\si_{a,b} \cdot x))}{d\mu(\si_{a,b} \cdot x)} = \frac{d \mu(g \cdot x)}{d\mu(x)} \; D_{ga,gb}(x_a,x_b) \; D_{a,b}(x_a,x_b)^{-1} \; .
\end{equation}
An important step in the proof of the lemma is to show that for `most' $g \in G$ tending to infinity, $D_{ga,gb}(x_0,x_1)$ is close to $1$.

{\bf Claim 1.} Fix $a,b \in G$. For every $i \in \{0,1\}$ and $\eps > 0$, the set
$$W_{i,\eps} = \Bigl\{ g \in G \Bigm| \; \Bigl| \sqrt{\frac{\mu_{ga}(i)}{\mu_{gb}(i)}} - 1 \Bigr| > \eps \Bigr\} \quad\text{satisfies}\quad \sum_{g \in W_{i,\eps}} \mu_{kgh}(i) < +\infty \quad\text{for all $k,h \in G$.}$$
Note that
\begin{align*}
\eps^2 \; \sum_{g \in W_{i,\eps}} \mu_{gb}(i) & \leq \sum_{g \in W_{i,\eps}} \mu_{gb}(i) \; \Bigl( \sqrt{\frac{\mu_{ga}(i)}{\mu_{gb}(i)}} - 1 \Bigr)^2 \\
&\leq \sum_{g \in G} \Bigl(\sqrt{\mu_{ga}(i)} - \sqrt{\mu_{gb}(i)}\Bigr)^2 < +\infty \; ,
\end{align*}
where the last inequality follows from the nonsingularity of the right Bernoulli action and the Kakutani criterion in \cite{Kak48}. By the nonsingularity of the left and the right Bernoulli action, we also have that
$$\sum_{g \in W_{i,\eps}} \Bigl(\sqrt{\mu_{kgh}(i)} - \sqrt{\mu_{gb}(i)}\Bigr)^2 \leq \sum_{g \in G} \Bigl(\sqrt{\mu_{kgh}(i)} - \sqrt{\mu_{gb}(i)}\Bigr)^2 < +\infty$$
for all $k,h \in G$. Since $(\sqrt{\mu_{gb}(i)})_{g \in W_{i,\eps}}$ belongs to $\ell^2(W_{i,\eps})$, also $(\sqrt{\mu_{kgh}(i)})_{g \in W_{i,\eps}}$ belongs to $\ell^2(W_{i,\eps})$ and claim 1 is proven.

{\bf Claim 2.} The action $G \actson (X,\mu)$ is either dissipative or conservative.

Denote by
$$\cD = \Bigl\{ x \in X \Bigm| \sum_{g \in G} \frac{d\mu(g \cdot x)}{d\mu(x)} < +\infty \Bigr\}$$
the dissipative part of $X$. We prove that $\cD$ is invariant under the action $\cS_G \actson (X,\mu)$. By \cite[Theorem 1.1]{AP77}, this last action is ergodic, so that claim 2 follows.

Fix $a,b \in G$. It suffices to prove that $\mu(\si_{a,b} \cdot \cD \vartriangle \cD) = 0$. Take $\eps = 1/2$ and define the sets $W_i = W_{i,1/2}$ as in claim 1. When $g \in G \setminus (W_0 \cup W_1)$, we have that $\mu_{ga}(i)/\mu_{gb}(i) \in [1/4,4]$ for all $i \in \{0,1\}$, so that $1/16 \leq D_{ga,gb}(x_0,x_1) \leq 16$ for all $x_0,x_1 \in \{0,1\}$. By \eqref{eq.basic-formula-flip}, the set of $x \in X$ satisfying
$$\sum_{g \in G \setminus (W_0 \cup W_1)} \frac{d\mu(g \cdot x)}{d\mu(x)} < +\infty$$
is $\si_{a,b}$-invariant. To conclude the proof of claim 2, it then suffices to prove that
\begin{equation}\label{eq.to-prove-claim-2}
\sum_{g \in W_i} \frac{d\mu(g \cdot x)}{d\mu(x)} < +\infty \quad\text{for all $i \in \{0,1\}$ and a.e.\ $x \in X$.}
\end{equation}
For $k \in G$ and $i \in \{0,1\}$, write $\cU_{k,i} = \{x \in X \mid x_k = i\}$. For every $k \in G$ and $i \in \{0,1\}$, we have
$$\int_X \sum_{g \in W_i} \frac{d\mu(g \cdot x)}{d\mu(x)} \; 1_{\cU_{k,i}}(x) \; d\mu(x) = \sum_{g \in W_i} \mu(g \cdot \cU_{k,i}) = \sum_{g \in W_i} \mu(\cU_{gk,i}) = \sum_{g \in W_i} \mu_{gk}(i) < +\infty \; ,$$
where the last inequality follows from claim 1. Since $\bigcup_{k \in G} \cU_{k,i}$ has a complement of measure zero, we conclude that \eqref{eq.to-prove-claim-2} holds. So claim 2 is proven.

For the rest of the proof, we may now assume that $G \actson (X,\mu)$ is conservative. Choose any pmp action $G \actson (Y,\eta)$ and consider the diagonal product $G \actson X \times \R \times Y$ with the Maharam extension of $G \actson (X,\mu)$. We have to prove that \eqref{eq.extra-invariance} holds for any $G$-invariant function in $L^\infty(X \times \R \times Y)$. The Maharam extension preserves the infinite measure $d\mu \times e^{-t} \, dt$. We replace this measure by an absolutely continuous probability measure, which is no longer invariant, but still has good regularity properties. By doing this, as we are working on a probability space, all bounded invariant functions are integrable and this fact is crucial for our method which makes use of ratio ergodic theorems.

Define the probability measure $\nu$ on $\R$ given by $d\nu(t) = (1/2) \exp(-|t|) \, dt$. Write $\mutil = \mu \times \nu$. Denote by $\|\,\cdot\,\|_1$ the $L^1$-norm on $L^\infty(X \times \R \times Y)$ w.r.t.\ the probability measure $\mutil \times \eta$. Denote by $E : L^\infty(X \times \R \times Y) \recht L^\infty(X \times \R \times Y)^G$ the unique conditional expectation preserving the probability measure $\mutil \times \eta$. Denote by $\cC \subset L^\infty(X \times \R \times Y)$ the set of Borel functions $F : X \times \R \times Y \recht [0,1]$ satisfying the following two properties.
\begin{itemlist}
\item There exists a finite subset $\cF \subset G$ such that $F \in L^\infty(\{0,1\}^\cF \times \R \times Y)$.
\item The function $F$ is uniformly continuous in the $\R$-variable.
\end{itemlist}
Since the linear span of $\cC$ is $\|\,\cdot\,\|_1$-dense in $L^\infty(X \times \R \times Y)$, it suffices to prove that for all $F \in \cC$, we have that
\begin{equation}\label{eq.goal}
|(E(F))(\si_{a,b} \cdot x , t + \be(\si_{a,b},x), y) - (E(F))(x,t,y) | \leq \exp(4\eps)(1+\eps) - 1
\end{equation}
for all $\eps > 0$, $a,b \in G$ and a.e.\ $(x,t,y) \in X \times \R \times Y$. Fix $F \in \cC$, $\eps > 0$ and $a,b \in G$. Take $0 < \delta < \eps$ such that $|F(x,t,y)-F(x,s,y)| \leq \eps$ for all $x \in X$, $y \in Y$ and $s,t \in \R$ with $|s-t| \leq \delta$.

For $i,j \in \{0,1\}$ and $h,k \in G$, write $R_{h,k}(i,j) = \log D_{h,k}(i,j)$. Define
$$V = \bigl\{g \in G \bigm| \; |R_{ga,gb}(i,j)| \leq \delta \;\;\text{for all $i,j \in \{0,1\}$}\;\bigr\} \; .$$

{\bf Claim 3.} For a.e.\ $(x,t) \in X \times \R$, we have
$$
\sum_{g \in G \setminus V} \frac{d\mutil(g \cdot (x,t))}{d\mutil(x,t)} < +\infty \; .
$$

Using the notation of claim 1, take $\eta > 0$ small enough such that $G \setminus V \subset W_{0,\eta} \cup W_{1,\eta}$. We now proceed as in the proof of \eqref{eq.to-prove-claim-2}. Fix $i \in \{0,1\}$, $k \in G$ and write $\cU_{k,i} = \{x \in X \mid x_k = i\}$. Then,
\begin{align*}
\int_{X \times \R} \sum_{g \in W_{i,\eta}} \frac{d\mutil(g \cdot (x,t))}{d\mutil(x,t)} \; & 1_{\cU_{k,i} \times \R}(x,t) \; d\mutil(x,t)
 = \sum_{g \in W_{i,\eta}} \mutil(g \cdot (\cU_{k,i} \times \R)) \\ &= \sum_{g \in W_{i,\eta}} \mutil(\cU_{gk,i} \times \R) = \sum_{g \in W_{i,\eta}} \mu(\cU_{gk,i}) = \sum_{g \in W_{i,\eta}} \mu_{gk}(i) < \infty \; ,
\end{align*}
by claim 1. Since $\bigcup_{k \in G} \cU_{k,i} \times \R$ has a complement of measure zero, we conclude that
$$\sum_{g \in W_{i,\eta}} \frac{d\mutil(g \cdot (x,t))}{d\mutil(x,t)} < \infty \quad\text{for all $i \in \{0,1\}$ and a.e.\ $(x,t) \in X \times \R$.}$$
Since $G \setminus V \subset W_{0,\eta} \cup W_{1,\eta}$, claim 3 is proven.

Applying \cite[Theorem A.1]{D18} to the nonsingular action $G \actson (X \times \R \times Y,\mutil \times \eta)$ and the function $F$, we find an increasing sequence of finite subsets $\cG_n \subset G$ with $\bigcup_n \cG_n = G$ such that
$$
\lim_n \frac{\dis \sum_{g \in \cG_n} \frac{d\mutil(g \cdot (x,t))}{d\mutil(x,t)} \; F(g \cdot (x,t,y))}{\dis \sum_{g \in \cG_n} \frac{d\mutil(g \cdot (x,t))}{d\mutil(x,t)}} = (E(F))(x,t,y)
$$
for a.e.\ $(x,t,y) \in X \times \R \times Y$. Take a finite subset $\cF \subset G$ such that $F \in L^\infty(\{0,1\}^\cF \times \R \times Y)$. Define $U = V \setminus (\cF a^{-1} \cup \cF b^{-1})$. Since $G \actson X \times \R \times Y$ is conservative, using claim 3 and the finiteness of $\cF$, we also get that
\begin{equation}\label{eq.Dan}
\lim_n \frac{\dis \sum_{g \in \cG_n \cap U} \frac{d\mutil(g \cdot (x,t))}{d\zeta(x,t)} \; F(g \cdot (x,t,y))}{\dis \sum_{g \in \cG_n \cap U} \frac{d\mutil(g \cdot (x,t))}{d\zeta(x,t)}} = (E(F))(x,t,y)
\end{equation}
for any probability measure $\zeta$ that is equivalent with $\mutil$ and a.e.\ $(x,t,y) \in X \times \R \times Y$. We use \eqref{eq.Dan} to estimate $(E(F))(\si_{a,b} \cdot x, t + \be(\si_{a,b},x),y)$. Recall that $\be(\si_{a,b},x) = R_{a,b}(x_a,x_b)$.

Write
$$\al(g,x) = \log \frac{d\mu(g \cdot x)}{d\mu(x)} \quad\text{so that}\quad g \cdot (x,t,y) = (g \cdot x, t + \al(g,x), g \cdot y) \; .$$
Note that $g \cdot (\si_{a,b} \cdot x) = \si_{ga,gb} \cdot (g \cdot x)$. When $g \in U$, we have $ga,gb \not\in \cF$. Using \eqref{eq.basic-formula-flip}, we get that
$$F(g \cdot (\si_{a,b} \cdot x,t+R_{a,b}(x_a,x_b),y)) = F(g \cdot x, t + \al(g,x) + R_{ga,gb}(x_a,x_b),g \cdot y)$$
for all $g \in U$ and a.e.\ $(x,t,y)$. Since $|R_{ga,gb}(x_a,x_b)| \leq \delta$ for all $g \in U \subset V$, we conclude that
\begin{equation}\label{eq.new-step-1}
|F(g \cdot (\si_{a,b} \cdot x, t + \be(\si_{a,b},x),y)) - F(g \cdot (x,t,y))| \leq \eps
\end{equation}
for all $g \in U$ and a.e.\ $(x,t,y)$.

Note that $d\nu(t+s)/d\nu(t) \leq \exp(|s|)$ for all $s,t \in \R$. Applying \eqref{eq.basic-formula-flip} twice, we get that
$$
\frac{d\mutil(g \cdot (\si_{a,b} \cdot x,t+\be(\si_{a,b},x)))}{d\mutil(x,t)} = \frac{d\mu(g \cdot x)}{d\mu(x)} \; D_{ga,gb}(x_a,x_b) \; \frac{d\nu(t + \al(g,x)+ R_{ga,gb}(x_a,x_b))}{d\nu(t)}
$$
so that for all $g \in U$ and all $(x,t) \in X \times \R$,
\begin{equation}\label{eq.new-step-2}
\exp(-2\eps) \, \frac{d\mutil(g \cdot (x,t))}{d\mutil(x,t)} \leq \frac{d\mutil(g \cdot (\si_{a,b} \cdot x,t+\be(\si_{a,b},x)))}{d\mutil(x,t)} \leq \exp(2\eps) \, \frac{d\mutil(g \cdot (x,t))}{d\mutil(x,t)} \; .
\end{equation}

Since $\|F\|_\infty \leq 1$, a combination of \eqref{eq.Dan}, \eqref{eq.new-step-1} and \eqref{eq.new-step-2} implies that \eqref{eq.goal} holds. This concludes the proof of the lemma.
\end{proof}

Using Lemma \ref{lem.automatic-invariance-permutation-group}, we immediately get the following dichotomy. Since for abelian groups $G$, the left Bernoulli shift by $g \in G$ equals the right Bernoulli shift by $g^{-1}$, the first part of Theorem \ref{thm-A} is a direct consequence of the following theorem.

\begin{theorem}\label{thm.dichotomy-dissipative-weakly-mixing}
Let $G$ be an amenable group and $G \actson (X,\mu) = \prod_{g \in G} (\{0,1\},\mu_g)$ any nonsingular Bernoulli action, with $\mu_g(0) \in (0,1)$ for all $g \in G$. Assume that $\mu$ is nonatomic and that also the right Bernoulli shift is nonsingular.

Then either $G \actson (X,\mu)$ is dissipative or $G \actson (X,\mu)$ is weakly mixing.
\end{theorem}
\begin{proof}
Assume that $G \actson (X,\mu)$ is not dissipative and that $G \actson (Y,\eta)$ is an arbitrary ergodic pmp action. Let $F \in L^\infty(X \times Y)$ be invariant under the diagonal action of $G$. In the context of Lemma \ref{lem.automatic-invariance-permutation-group}, we view $F \in L^\infty(X \times \R \times Y)$ as a $G$-invariant function that does not depend on the $\R$-variable. It follows from Lemma \ref{lem.automatic-invariance-permutation-group} that $F(\si \cdot x,y) = F(x,y)$ for all $\si \in \cS_G$ and a.e.\ $(x,y) \in X \times Y$. By \cite[Theorem 1.1]{AP77}, the action $\cS_G \actson (X,\mu)$ is ergodic. Therefore, $F \in 1 \ot L^\infty(Y)$. Since $G \actson (Y,\eta)$ is ergodic, it follows that $F$ is essentially constant. So, $G \actson X \times Y$ is ergodic, meaning that $G \actson (X,\mu)$ is weakly mixing.
\end{proof}

We now use Lemma \ref{lem.automatic-invariance-permutation-group} to prove the following slightly more precise version of Theorem \ref{thm-B}. Recall that a nonsingular ergodic action $G \actson (X,\mu)$ is said to be of \emph{stable type $T$} if for any ergodic pmp action $G \actson (Y,\eta)$, the diagonal action $G \actson X \times Y$ is of type $T$.

\begin{theorem}\label{thm.type-abelian}
Let $G$ be an infinite abelian group and $G \actson (X,\mu) = \prod_{g \in G} (\{0,1\},\mu_g)$ any nonsingular Bernoulli action, with $\mu_g(0) \in (0,1)$ for all $g \in G$. Assume that $G$ is not locally finite and denote by $\cL \subset [0,1]$ the set of limit values of $\mu_g(0)$ for $g \recht \infty$. Assume that $\mu$ is nonatomic and that $G \actson (X,\mu)$ is not dissipative.

Then, $G \actson (X,\mu)$ is either of stable type II$_1$ or of stable type III, but never of type II$_\infty$. More precisely, exactly one of the following statements holds.
\begin{enumlist}
\item $\cL = \{\lambda\}$ with $\lambda \in (0,1)$ and $\sum_{g \in G} (\mu_g(0) - \lambda)^2 < +\infty$. Then $\mu \sim \nu^G$ where $\nu(0) = \lambda$ and $G \actson (X,\mu)$ is of stable type II$_1$.
\item $\cL = \{\lambda\}$ with $\lambda \in (0,1)$ and $\sum_{g \in G} (\mu_g(0) - \lambda)^2 = +\infty$. Then $G \actson (X,\mu)$ is of stable type III$_1$.
\item $\cL$ contains at least two points. Then $\cL \subset [0,1]$ is a perfect set and $G \actson (X,\mu)$ is of stable type III$_1$.
\item $\cL = \{0\}$ or $\cL = \{1\}$. Then $G \actson (X,\mu)$ is of stable type III.
\end{enumlist}
\end{theorem}

In particular, if $G$ is a non locally finite abelian group, only singletons and perfect sets arise as the set of limit points of $\mu_g(0)$ for a nonsingular conservative Bernoulli action. In the locally finite case, the situation is different: every closed subset of $[0,1]$ may arise as the set of limit points and there are Bernoulli actions of type II$_\infty$. We prove this in Propositions \ref{prop.conservative-locally-finite} and \ref{prop.locally-finite-type-II-infty} below.

\begin{proof}
We first deduce from the nonsingularity of $G \actson (X,\mu)$ that the following holds: if $\lambda \in \cL$ is an isolated point, then $\cL = \{\lambda\}$.

Take $\eps > 0$ such that $\cL \cap [\lambda-\eps,\lambda+\eps] = \{\lambda\}$. Define the infinite subset $W \subset G$ given by $W = \{g \in G \mid |\mu_g(0) - \lambda| \leq \eps\}$. By Kakutani's criterion \eqref{eq.criterion-nonsingular}, we have for every $g \in G$ that $\lim_{h \recht \infty} |\mu_{gh}(0) - \mu_h(0)| = 0$. It follows that $|g W \vartriangle W| < +\infty$ for every $g \in G$, meaning that $W \subset G$ is an almost invariant subset. If $G \setminus W$ is finite, we conclude that $\lim_{g \recht \infty} \mu_g(0) = \lambda$. If $G \setminus W$ is infinite, then $W \subset G$ is a nontrivial almost invariant set, meaning that the group $G$ has more than one end. By the version of Stallings' theorem for possibly infinitely generated groups (see \cite[Theorem IV.6.10]{DD89} and Remark \ref{rem.ends-groups} below), since $G$ is abelian and not locally finite, this implies that $G$ is virtually cyclic.

Choose a finite index copy of $\Z$ inside $G$. Then, $W \cap \Z$ is a nontrivial almost invariant subset of $\Z$. So, $W \cap \Z$ contains either $[n_0,+\infty)$ or $(-\infty,-n_0]$ for large enough $n_0 \in \N$. We assume that $[n_0,+\infty) \subset W$ and the other case is handled analogously. Taking $n_0$ large enough, we find a $\delta > 0$ such that
$$(\sqrt{\mu_a(0)} - \sqrt{\mu_b(0)})^2 \geq \delta \quad\text{whenever $a \leq -n_0$ and $b \geq n_0$.}$$
Essentially repeating the computation in \cite[Proposition 4.1]{VW17}, we have for every $a \leq -2n_0$,
\begin{align*}
\int_X \sqrt{\frac{d\mu(a \cdot x)}{d\mu(x)}} \, d\mu(x) & \leq \prod_{b \in \Z} \bigl( \sqrt{\mu_{a+b}(0) \, \mu_{b}(0)} + \sqrt{\mu_{a+b}(1) \, \mu_{b}(1)} \bigr) \\
& \leq \exp\Bigl(- \frac{1}{2} \sum_{b \in \Z} \bigl( (\sqrt{\mu_{a+b}(0)} - \sqrt{\mu_b(0)})^2 + (\sqrt{\mu_{a+b}(1)} - \sqrt{\mu_b(1)})^2 \bigr) \Bigr)\\
& \leq \exp\Bigl(- \frac{1}{2} \sum_{b =n_0+1}^{-a-n_0+1} (\sqrt{\mu_{a+b}(0)} - \sqrt{\mu_b(0)})^2 \Bigr) \\
& \leq \exp(\delta n_0) \, \exp(\delta a / 2) \; .
\end{align*}
It follows that
$$\int_X \sum_{a=-\infty}^{-2n_0} \sqrt{\frac{d\mu(a \cdot x)}{d\mu(x)}} \, d\mu(x) < +\infty \; .$$
By \cite[Lemma 2.2]{Kos12} and \cite[Proposition 1.3.2]{Aar97}, the action $\Z \actson (X,\mu)$ is dissipative. Since $\Z$ has finite index in $G$, also $G \actson (X,\mu)$ is dissipative, contrary to our assumptions.

So we have proven that $\cL$ is either a singleton or a perfect set. Choose any ergodic pmp action $G \actson (Y,\eta)$ and consider the diagonal action $G \actson X \times \R \times Y$ with the Maharam extension of $G \actson (X,\mu)$. Let $F \in L^\infty(X \times \R \times Y)$ be a $G$-invariant function that generates the fixed point algebra $L^\infty(X \times \R \times Y)^G$. By Lemma \ref{lem.automatic-invariance-permutation-group}, $F$ is invariant under the Maharam extension of $\cS_G \actson (X,\mu)$.

If $\cL$ is a perfect set, it follows from the second point of Proposition \ref{prop.III-1-permutation-action} that $\cS_G \actson (X,\mu)$ is of type III$_1$. Therefore, $F$ only depends on the $Y$-variable. Since $G \actson (Y,\eta)$ is ergodic, it follows that $F$ is essentially constant. So, $G \actson (X,\mu)$ is of stable type III$_1$.

When $\cL = \{\lambda\}$ with $\lambda \in (0,1)$, we get that $\lim_{g \recht \infty} \mu_g(0) = \lambda$. If $\sum_{g \in G} (\mu_g(0) - \lambda)^2 < +\infty$, we have that $\mu \sim \nu^G$ where $\nu(0) = \lambda$. In that case, $G \actson (X,\mu)$ is of stable type II$_1$. If $\sum_{g \in G} (\mu_g(0) - \lambda)^2 = +\infty$, the first point of Proposition \ref{prop.III-1-permutation-action} says that $\cS_G \actson (X,\mu)$ is of type III$_1$ and we again conclude that $G \actson (X,\mu)$ is of stable type III$_1$.

It remains to consider the cases $\cL = \{0\}$ and $\cL = \{1\}$. By symmetry, we may assume that $\lim_{g \recht \infty} \mu_g(0) = 0$. By \cite[Theorem 1.2]{SV77}, the action $\cS_G \actson (X,\mu)$ is of type III (see the discussion in Section \ref{sec.permutation-action}). Denote by $\R \actson^\gamma (Z,\zeta)$ the flow associated to $\cS_G \actson (X,\mu)$. Since $F$ is invariant under the Maharam extension of $\cS_G \actson (X,\mu)$, it follows that the flow associated to $G \actson (X,\mu)$ is a factor of the flow $\R \actson Z \times Y$ given by $t \cdot (z,y) = (\gamma_t(z),y)$. So, the flow associated to $G \actson (X,\mu)$ cannot be the translation action of $\R$ on $\R$ and we conclude that $G \actson (X,\mu)$ is of stable type III.
\end{proof}

In the locally finite case, the situation is quite different.

\begin{theorem}\label{thm.type-abelian-locally-finite}
Let $G \actson (X,\mu) = \prod_{g \in G} (\{0,1\},\mu_g)$ be any nonsingular Bernoulli action of any abelian infinite locally finite group $G$. Assume that $\mu$ is nonatomic and that $G \actson (X,\mu)$ is not dissipative.

\begin{enumlist}
\item The action $G \actson (X,\mu)$ is of type II$_1$ if and only if there exists a $\lambda \in (0,1)$ such that $\sum_{g \in G} (\mu_g(0) - \lambda)^2 < +\infty$.
\item The action $G \actson (X,\mu)$ is of type II$_\infty$ if and only if there exists an infinite almost invariant subset $W \subset G$ with infinite complement $V = G \setminus W$ and a $\lambda \in (0,1)$ such that
    \begin{equation}\label{eq.condition-II-infty}
    \sum_{g \in W} (\mu_g(0) - \lambda)^2 + \sum_{g \in V} \min\{\mu_g(0),\mu_g(1)\} < +\infty \; .
    \end{equation}
\item In all other cases, $G \actson (X,\mu)$ is of type III.
\end{enumlist}
\end{theorem}

In Proposition \ref{prop.locally-finite-type-II-infty} below, we show that each of the cases in Theorem \ref{thm.type-abelian-locally-finite} does occur, including type II$_\infty$, for any infinite locally finite group $G$.

\begin{proof}
By Lemma \ref{lem.automatic-invariance-permutation-group}, the flow associated to $G \actson (X,\mu)$ is an $\R$-factor of the flow associated to $\cS_G \actson (X,\mu)$. So whenever $\cS_G \actson (X,\mu)$ is of type III, also $G \actson (X,\mu)$ is of type III.

Denote by $\cL \subset [0,1]$ the set of limit values of $\mu_g(0)$ as $g \recht \infty$. If $\cL \subset \{0,1\}$, it follows from \cite[Theorem 1.2]{SV77} that $\cS_G \actson (X,\mu)$ is of type III (see the discussion in Section \ref{sec.permutation-action}). If $\cL \cap (0,1)$ contains at least two points, it follows from the second point of Proposition \ref{prop.III-1-permutation-action} that $\cS_G \actson (X,\mu)$ is of type III.

It remains to consider the case where $\cL \cap (0,1) = \{\lambda\}$. Take $\eps > 0$ such that $0 < \eps < \lambda < 1-\eps < 1$. Define $W = \{g \in G \mid \eps < \mu_g(0) < 1-\eps\}$. Write $V = G \setminus W$. Since $0$, $1$ and $\lambda$ are the only possible limit values of $\mu_g(0)$ as $g \recht \infty$, it follows from Kakutani's criterion \eqref{eq.criterion-nonsingular} that $W$ and $V$ are almost invariant subsets of $G$. If the left hand side of \eqref{eq.condition-II-infty} equals $+\infty$, it again follows from \cite[Theorem 1.2]{SV77} that $\cS_G \actson (X,\mu)$ is of type III.

So we may assume that \eqref{eq.condition-II-infty} holds. If $V$ is finite, it follows that also $\sum_{g \in G} (\mu_g(0) - \lambda)^2 < +\infty$. Defining the probability measure $\nu$ on $\{0,1\}$ with $\nu(0) = \lambda$, we find that $\mu \sim \nu^G$, so that $G \actson (X,\mu)$ is of type II$_1$.

Finally assume that \eqref{eq.condition-II-infty} holds with $V$ infinite. We prove that $G \actson (X,\mu)$ is of type II$_\infty$. By Theorem \ref{thm.dichotomy-dissipative-weakly-mixing}, the action $G \actson (X,\mu)$ is ergodic. Partition $V = A \sqcup B$ such that
\begin{equation}\label{eq.atom-A-B}
\sum_{g \in A} \mu_g(1) + \sum_{g \in B} \mu_g(0) < +\infty \; .
\end{equation}
Choose a decreasing subsequence $V_n \subset V$ such that $|V \setminus V_n| = n$ and such that $\bigcap_n V_n = \emptyset$, which is possible because $V$ is infinite. Then define the increasing sequence of subsets $X_n \subset X$ given by
$$X_n = \{x \in X \mid x_g = 0 \;\;\text{for all $g \in V_n \cap A$ and}\;\; x_g = 1 \;\;\text{for all $g \in V_n \cap B$}\} \; .$$
By \eqref{eq.atom-A-B}, the union $\bigcup_n X_n$ has a complement of measure zero. Denote by $\nu$ the probability measure on $\{0,1\}$ given by $\nu(0) = \lambda$. Identifying $X_n$ with $\{0,1\}^{G \setminus V_n}$, we define the probability measure $\nu_n = \nu^{G \setminus V_n}$ on $X_n$. By construction, the restriction of $\mu$ to $X_n$ is equivalent with $\nu_n$.

Denote by $\cR = \cR(G \actson X)$ the orbit equivalence relation of $G \actson (X,\mu)$. Since $G \actson (X,\mu)$ is ergodic and $\mu$ is nonatomic, the equivalence relation $\cR$ is ergodic and not of type I. So also the restriction $\cR|_{X_n}$ is ergodic and not of type I. We claim that the probability measure $\nu_n$ is invariant under $\cR|_{X_n}$. Once this claim is proven, note that $\nu_{n+m}(X_n) = \lambda^m$. Since $\bigcup_n X_n$ has a complement of measure zero, it then follows that $\cR$ is of type II$_\infty$.

To prove the claim, fix $n \in \N$ and fix a finite subgroup $\Lambda < G$. It suffices to prove that $\cR(\Lambda \actson X)|_{X_n}$ preserves the measure $\nu_n$.
Define $A_0 \subset A$, $B_0 \subset B$ and $Y \supset X_n$ given by
\begin{align*}
& A_0 = \bigcap_{g \in \Lambda} g (V_n \cap A) \quad , \quad B_0 = \bigcap_{g \in \Lambda} g (V_n \cap B) \quad\text{and}\\ & Y = \{x \in X \mid x_g = 0 \;\;\text{for all $g \in A_0$ and}\;\; x_g = 1 \;\;\text{for all $g \in B_0$}\} \; .
\end{align*}
Since $A$ and $B$ are almost invariant subsets of $G$ and $V \setminus V_n$ is finite, we get that $A \setminus A_0$ and $B \setminus B_0$ are finite sets. Write $U = G \setminus (A_0 \cup B_0)$ and identify $Y = \{0,1\}^{U}$. Define the probability measure $\zeta = \nu^{U}$ on $Y$. Since $A_0$, $B_0$ and $U$ are globally $\Lambda$-invariant, we get that $Y$ is $\Lambda$-invariant and that $\zeta$ is a $\Lambda$-invariant probability measure. The restriction of $\zeta$ to $X_n$ equals $\lambda^{|V_n \setminus (A_0 \cup B_0)|} \, \nu_n$. So, $\nu_n$ is invariant under $\cR(\Lambda \actson X)|_{X_n}$ and the claim is proven.
\end{proof}

In the following result, we prove that for many Bernoulli actions of locally finite groups, the type is the same as the type for the associated permutation actions. We then use this in Proposition \ref{prop.locally-finite-type-II-infty} below to prove that all possible types may occur.

\begin{proposition}\label{prop.conservative-locally-finite}
Let $G$ be an infinite locally finite group and $(\lambda_n)_{n \in \N}$ any sequence in $(0,1)$.
\begin{enumlist}
\item There exists a strictly increasing sequence of subgroups $G_n \subset G$ with the following properties: $\bigcup_n G_n = G$ and, writing $\mu_g(0) = \lambda_n$ for all $g \in G_n \setminus G_{n-1}$, both the left and the right Bernoulli action of $G$ on $(X,\mu) = \prod_{g \in G} (\{0,1\},\mu_g)$ are nonsingular and conservative.
\item Whenever $G_n \subset G$ is a strictly increasing sequence of subgroups with $\bigcup_n G_n = G$ satisfying the conclusions of point 1, the Maharam extension $G \actson X \times \R$ of the Bernoulli action and the Maharam extension $\cS_G \actson X \times \R$ of the permutation action have the same fixed point algebra: $L^\infty(X \times \R)^G = L^\infty(X \times \R)^{\cS_G}$.
\end{enumlist}
\end{proposition}
\begin{proof}
1.\ When $g \in G_n$ and $k \in G \setminus G_n$, we have $\mu_{gk} = \mu_k$. So for any choice of $G_n$, the action $G \actson (X,\mu)$ is nonsingular. Although the value of $\mu_e(0)$ is irrelevant, to avoid confusion, we assume that $G_0 = \emptyset$ and that $G_n$ is a finite subgroup for every $n \geq 1$. We define $\mu_g(0) = \lambda_n$ for all $g \in G_n \setminus G_{n-1}$ and all $n \geq 1$.

For any finite subset $\cF \subset G$, using the convexity of $a \mapsto a^{-1}$, we have
\begin{multline}\label{eq.a-conservative-estimate}
\int_X \Bigl(\sum_{g \in G} \frac{d\mu(g \cdot x)}{d\mu(x)} \Bigr)^{-1} \, d\mu(x) \leq \int_X \Bigl(\sum_{g \in \cF} \frac{d\mu(g \cdot x)}{d\mu(x)} \Bigr)^{-1} \, d\mu(x) \\ \leq \frac{1}{|\cF|^2} \sum_{g \in \cF} \int_X \frac{d\mu(x)}{d\mu(g \cdot x)} \, d\mu(x) \; .
\end{multline}
Given any choice of $G_n$, we have for every $g \in G_n \setminus G_{n-1}$ that
\begin{align*}
\int_X \frac{d\mu(x)}{d\mu(g \cdot x)} \, d\mu(x) & = \prod_{k \in G_n} \Bigl( \frac{\mu_k(0)^2}{\mu_{gk}(0)} + \frac{\mu_k(1)^2}{\mu_{gk}(1)} \Bigr) \\
& = \prod_{m=1}^{n-1} \Bigl( \prod_{\substack{k \in (G_m \setminus G_{m-1}) \\ \cup g^{-1}(G_m \setminus G_{m-1})}} \Bigl( \frac{\mu_k(0)^2}{\mu_{gk}(0)} + \frac{\mu_k(1)^2}{\mu_{gk}(1)} \Bigr) \Bigr) \\
& = \prod_{m=1}^{n-1} \Bigl( \Bigl( \frac{\lambda_m^2}{\lambda_n} + \frac{(1-\lambda_m)^2}{1-\lambda_n} \Bigr)^{|G_m \setminus G_{m-1}|} \; \cdot \;
\Bigl( \frac{\lambda_n^2}{\lambda_m} + \frac{(1-\lambda_n)^2}{1-\lambda_m} \Bigr)^{|G_m \setminus G_{m-1}|}\Bigr) \; .
\end{align*}

Since this last expression only depends on the sequence $\lambda_n$ and the cardinalities $|G_m \setminus G_{m-1}|$ with $m \leq n-1$, we can inductively choose $G_n$ large enough such that
$$\frac{1}{|G_n \setminus G_{n-1}|^2} \sum_{g \in G_n \setminus G_{n-1}} \int_X \frac{d\mu(x)}{d\mu(g \cdot x)} \, d\mu(x) < \frac{1}{n} \; .$$
It then follows from \eqref{eq.a-conservative-estimate} that
$$\int_X \Bigl(\sum_{g \in G} \frac{d\mu(g \cdot x)}{d\mu(x)} \Bigr)^{-1} \, d\mu(x) = 0$$
so that $G \actson (X,\mu)$ is conservative.

2.\ When $g \in G_n$ and $k \in G \setminus G_n$, we have $\mu_{kg} = \mu_k$, so that also the right Bernoulli action on $(X,\mu)$ is nonsingular. Thus by Lemma \ref{lem.automatic-invariance-permutation-group}, we have $L^\infty(X \times \R)^G \subset L^\infty(X \times \R)^{\cS_G}$. To prove the converse, assume that $F \in L^\infty(X \times \R)^{\cS_G}$. Fix $n_0 \in \N$ and $g_0 \in G_{n_0}$. For every $n \geq n_0$, define the finite permutation
$$\si_n : G \recht G : \si_n(k) = g_0 k \;\;\text{if $k \in G_n$, and}\;\; \si_n(k) = k \;\;\text{if $k \not\in G_n$.}$$
For $n \geq n_0$, also define the measure preserving transformation $\zeta_n : X \recht X$ given by
$$(\zeta_n(x))_k = x_k \;\;\text{if $k \in G_n$, and}\;\; (\zeta_n(x))_k = x_{g_0^{-1}k} \;\;\text{if $k \not\in G_n$.}$$
By construction, $g_0 \cdot x = \si_n \cdot (\zeta_n(x))$. Also,
$$\al(g_0,x) = \log\Bigl(\prod_{k \in G_n} \frac{\mu_{g_0 k}(x_k)}{\mu_k(x_k)} \Bigr) = \beta(\si_n,\zeta_n(x)) \; .$$
Since $F$ is $\cS_G$-invariant, we find that
$$F(g_0 \cdot x, t + \al(g_0,x)) = F(\si_n \cdot (\zeta_n(x)) , t+ \beta(\si_n,\zeta_n(x))) = F(\zeta_n(x),t)$$
for all $n \geq n_0$ and a.e.\ $(x,t) \in X \times \R$. When $n \recht +\infty$, we have $F(\zeta_n(x),t) \recht F(x,t)$ weakly$^*$. So we conclude that $F$ is $g_0$-invariant. Since this holds for all $g_0 \in G_{n_0}$ and all $n_0 \in \N$, the proposition is proven.
\end{proof}

\begin{proposition}\label{prop.locally-finite-type-II-infty}
Let $G$ be an infinite, locally finite group. Then $G$ admits weakly mixing nonsingular Bernoulli actions of any possible type: II$_1$, II$_\infty$ and III$_\lambda$, for any $\lambda \in [0,1]$.
\end{proposition}
\begin{proof}
Obviously, $G$ admits a pmp Bernoulli action of type II$_1$. Taking $\lambda \in (0,1)$ and applying Proposition \ref{prop.conservative-locally-finite} to the sequence $1/2$, $(1+\lambda)^{-1}$, $1/2$, $(1+\lambda)^{-1}, \ldots$, we obtain a conservative Bernoulli action with the following properties. By Theorem \ref{thm.dichotomy-dissipative-weakly-mixing}, the action $G \actson (X,\mu)$ is weakly mixing. By Proposition \ref{prop.III-1-permutation-action}, $\lambda$ belongs to the ratio set of $\cS_G \actson (X,\mu)$. By construction, the Radon-Nikodym derivative of $\cS_G \actson (X,\mu)$ only takes values that are powers of $\lambda$. So, $\cS_G \actson (X,\mu)$ is of type III$_\lambda$. By Proposition \ref{prop.conservative-locally-finite}, also $G \actson (X,\mu)$ is of type III$_\lambda$.

Choosing $\lambda_1,\lambda_2 \in (0,1)$ generating a dense subgroup of $\R^*_+$, applying the same reasoning to the sequence $1/2$, $(1+\lambda_1)^{-1}$, $(1+\lambda_2)^{-1}$, $1/2, \ldots$, we obtain a weakly mixing nonsingular Bernoulli action $G \actson (X,\mu)$ of type III$_1$.

Next, fix $\lambda \in (0,1)$ and consider the sequence $\lambda_n = \lambda^{2^n}$. By Proposition \ref{prop.conservative-locally-finite}, we find a strictly increasing sequence of subgroups $G_n < G$ so that the associated Bernoulli action is nonsingular and conservative. From the proof of Proposition \ref{prop.conservative-locally-finite}, we see that we can find such $G_n$ with $|G_n \setminus G_{n-1}|$ growing arbitrarily fast. So by \cite{HO83} (see also \cite[Section 3]{GSW84}), we can make this choice so that the homoclinic equivalence relation $\cR_G$ on $(X,\mu)$ is of type III$_0$ (see Section \ref{sec.permutation-action} for the definition of $\cR_G$). By Proposition \ref{prop.conservative-locally-finite}, it follows that
\begin{equation}\label{eq.theflow}
L^\infty(X \times \R)^{\cR_G} \subset L^\infty(X \times \R)^{\cS_G} = L^\infty(X \times \R)^G \; .
\end{equation}
By Theorem \ref{thm.dichotomy-dissipative-weakly-mixing}, $G \actson (X,\mu)$ is weakly mixing. By \eqref{eq.theflow}, the (ergodic) flow associated to $G \actson (X,\mu)$ admits a properly ergodic flow of $\R$ as an $\R$-factor. So also $G \actson (X,\mu)$ is of type III$_0$.

Finally, we construct an example of type II$_\infty$. Fix $\lambda \in (0,1)$. Inductively choose a strictly increasing sequence of finite subgroups $G_n < G$ such that
$$\bigcup_n G_n = G \quad\text{and}\quad \bigl(1 - \lambda^{|G_{2n}|}\bigr)^{[G_{2n+1}:G_{2n}]-1} < \frac{1}{n^2} \quad\text{for all $n \in \N$.}$$
Note that this only imposes to choose $G_{2n+1}$ much larger than $G_{2n}$. At the even steps, we just require $G_{2n}$ to be strictly larger than $G_{2n-1}$.

Next choose a sequence $\gamma_n \in (0,1)$ tending to zero sufficiently fast such that
$$\sum_{n=1}^\infty \gamma_n \, |G_{2n} \setminus G_{2n-1}| < +\infty \; .$$
Define the probability measures $\mu_g$ on $\{0,1\}$ given by $\mu_g(0) = \gamma_n$ if $g \in G_{2n} \setminus G_{2n-1}$ for some $n \geq 1$, while $\mu_g(0) = 1-\lambda$ if $g \in G_{2n+1} \setminus G_{2n}$ for some $n \geq 0$. The associated Bernoulli action $G \actson (X,\mu)$ is nonsingular. We prove that it is weakly mixing and of type II$_\infty$.

Consider the orbit equivalence relation $\cR = \cR(G \actson X)$. Define the increasing sequence of subsets $X_n \subset X$ given by
$$X_n = \{x \in X \mid x_k = 1 \;\;\text{for all $k \in G_{2m} \setminus G_{2m-1}$ and all $m \geq n$}\} \; .$$
Then, $\mu(X_1) > 0$ and $\mu(X_n) \recht 1$. We denote $\nu = \mu(X_1)^{-1} \, \mu|_{X_1}$. The argument at the end of the proof of Theorem \ref{thm.type-abelian-locally-finite} shows that the restriction of $\cR$ to $X_1$ preserves the probability measure $\nu$. Below, we prove that the equivalence relation $\cR|_{X_1}$ has infinite orbits a.e.\ and that $G \cdot X_1$ has a complement of measure zero. Since the equivalence relation $\cR|_{X_1}$ is probability measure preserving, it follows that $G \actson (X,\mu)$ is conservative. By Theorem \ref{thm.dichotomy-dissipative-weakly-mixing}, $G \actson (X,\mu)$ is weakly mixing. In particular, $G \actson (X,\mu)$ is ergodic and thus of type II$_\infty$.

For every $n \geq 1$, define the subsets $Y_n \subset X$ given by
$$Y_n = \{x \in X \mid \text{there exists a $g \in G_{2n+1} \setminus G_{2n}$ such that $x_{gk} = 1$ for all $k \in G_{2n}$}\} \; .$$
Writing $G_{2n+1} \setminus G_{2n}$ as the disjoint union of $[G_{2n+1} : G_{2n}]-1$ left cosets of $G_{2n}$, it follows that
$$\nu(Y_n) = 1 - \bigl(1 - \lambda^{|G_{2n}|}\bigr)^{[G_{2n+1}:G_{2n}]-1} > 1 - \frac{1}{n^2} \; .$$
When $x \in Y_n \cap X_{n+1}$, we can take a $g \in G_{2n+1} \setminus G_{2n}$ such that $x_{gk} = 1$ for all $k \in G_{2n}$. Since $x \in X_{n+1}$ and since for all $m \geq n+1$, we have $g(G_{2m} \setminus G_{2m-1}) = G_{2m} \setminus G_{2m-1}$, it follows that also $x_{gk}=1$ for all $k \in G_{2m} \setminus G_{2m-1}$ and all $m \geq n+1$. So, $g^{-1} \cdot x \in X_1$. This means in particular that $Y_n \cap X_{n+1} \subset G \cdot X_1$ for all $n \geq 1$. It follows that $G \cdot X_1$ has a complement of measure zero.

By the Borel-Cantelli lemma, almost every $x \in X_1$ has the property that $x \in Y_n$ for all large enough $n$. When $x \in X_1 \cap \bigcap_{n=n_0}^\infty Y_n$, the reasoning in the previous paragraph provides for every $n \geq n_0$ an element $g_n \in G_{2n+1} \setminus G_{2n}$ such that $g_n^{-1} \cdot x \in X_1$. By Lemma \ref{lem.essential-free}, the action $G \actson (X,\mu)$ is essentially free. We conclude that $\cR|_{X_1}$ has infinite orbits a.e.
\end{proof}

\section{Strongly conservative actions}\label{sec.strongly-conservative}

\begin{lemma}\label{lem.ucp}
Let $G$ be a countable group and $G \actson (X,\mu)$ a nonsingular action on a standard probability space $(X,\mu)$. Let $\eta$ be a probability measure on $G$. Then the map
$$\theta_\eta : L^\infty(X,\mu) \recht L^\infty(X,\mu) : \theta_\eta(F) = \sum_{g \in G} \eta(g) \; \frac{\dis\sum_{k \in G} \eta(gk^{-1}) \, \frac{d\mu(k \cdot x)}{d \mu(x)} \, F(k \cdot x)}{\dis\sum_{k \in G} \eta(gk^{-1}) \, \frac{d\mu(k \cdot x)}{d \mu(x)}}$$
is unital, positive and measure preserving, meaning that $\int_X \theta_\eta(F) \, d\mu = \int_X F \, d\mu$.
\end{lemma}
\begin{proof}
Consider the probability measure $\eta \times \mu$ on $G \times X$. Define the probability measure $\nu$ on $X$ given by
$$\int_X H(x) \, d \nu(x) = \sum_{g \in G} \eta(g) \, \int_X H(g \cdot x) \, d\mu(x) \; .$$
Write $A = L^\infty(G \times X,\eta \times \mu)$. Denote $A_1 = L^\infty(X,\nu)$ and $A_2 = L^\infty(X,\mu)$. Equip $A,A_1,A_2$ with the traces given by $\eta \times \mu, \nu,\mu$. Then,
$$\pi_1 : A_1 \recht A : (\pi_1(F))(g,x) = F(g \cdot x) \quad\text{and}\quad \pi_2 : A_2 \recht A : (\pi_2(F))(g,x) = F(x)$$
are trace preserving, unital $*$-homomorphisms. Denote by $E_i : A \recht \pi_i(A_i)$ the unique trace preserving conditional expectations. Then,
$$\theta_\eta = \pi_2^{-1} \circ E_2 \circ E_1 \circ \pi_2 \; .$$
\end{proof}

Motivated by Lemma \ref{lem.ucp}, we introduce the following ad hoc definition.

\begin{definition}\label{def.strong-conservative}
Let $G$ be a countable group and $G \actson (X,\mu)$ a nonsingular action on a standard probability space $(X,\mu)$. We say that a sequence of probability measures $\eta_n$ on $G$ is strongly recurrent if
\begin{equation}\label{eq.strongly-conservative}
\sum_{g \in G} \eta_n(g)^2 \, \int_X \Bigl( \sum_{k \in G} \eta_n(gk^{-1}) \, \frac{d \mu(k \cdot x)}{d \mu(x)} \Bigr)^{-1} \; d \mu(x) \recht 0 \;\; .
\end{equation}
We say that $G \actson (X,\mu)$ is \emph{strongly conservative} if such a strongly recurrent sequence of probability measures exists.
\end{definition}

\begin{proposition}\label{prop.characterize-strong-conservative}
Let $G$ be a countable group and $G \actson (X,\mu)$ a nonsingular action on a standard probability space $(X,\mu)$.
\begin{enumlist}
\item If $G \actson (X,\mu)$ is strongly conservative, then $G \actson (X,\mu)$ is conservative.
\item If $G$ is amenable and $G \actson (X,\mu)$ is conservative, then the uniform probability measures $\eta_n$ on a right F{\o}lner sequence $\cF_n \subset G$ are strongly recurrent, so that $G \actson (X,\mu)$ is strongly conservative.
\item A sequence of probability measures $\eta_n$ on $G$ satisfying
\begin{equation}\label{eq.est-2}
\sum_{k \in G} \; \Bigl(\sum_{g \in G} \eta_n(g)^2 \, \eta_n(gk^{-1}) \Bigr) \; \int_X \frac{d\mu(x)}{d\mu(k \cdot x)} \; d\mu(x) \recht 0 \; ,
\end{equation}
is strongly recurrent.
\end{enumlist}
\end{proposition}
\begin{proof}
1. For every probability measure $\eta$ on $G$, write $\supp(\eta) = \{g \in G \mid \eta(g) > 0\}$. By convexity of the function $t \mapsto t^{-1}$, we have for every probability measure $\eta$ on $G$ that
\begin{align*}
\sum_{g \in G} \eta(g)^2 \, \Bigl(\sum_{k \in G} \eta(gk^{-1}) \, \frac{d\mu(k\cdot x)}{d\mu(x)}\Bigr)^{-1} &=
\sum_{g \in \supp(\eta)} \eta(g) \, \Bigl(\eta(g)^{-1} \, \sum_{k \in G} \eta(gk^{-1}) \, \frac{d\mu(k\cdot x)}{d\mu(x)}\Bigr)^{-1} \\[1ex]
& \geq \Bigl(\sum_{g \in \supp(\eta)} \eta(g) \, \eta(g)^{-1} \, \sum_{k \in G} \eta(gk^{-1}) \, \frac{d\mu(k\cdot x)}{d\mu(x)}\Bigr)^{-1} \\[1ex]
&\geq \Bigl( \sum_{k \in G} \frac{d\mu(k\cdot x)}{d\mu(x)}\Bigr)^{-1} \; .
\end{align*}
If $G \actson (X,\mu)$ is strongly conservative, with strongly recurrent sequence $\eta_n$, we take $\eta = \eta_n$ and integrate over $X$. It follows that
$$\Bigl( \sum_{k \in G} \frac{d\mu(k\cdot x)}{d\mu(x)}\Bigr)^{-1} = 0$$
for a.e.\ $x \in X$, so that $G \actson (X,\mu)$ is conservative.

2. Let $\cF_n \subset G$ be a right F{\o}lner sequence and define $\eta_n$ as the uniform probability measure on $\cF_n$. Let $\eps > 0$. Since $G \actson (X,\mu)$ is conservative, we can fix a finite subset $\cL \subset G$ such that
$$\int_X \Bigl( \sum_{k \in \cL} \frac{d\mu(k \cdot x)}{d\mu(x)} \Bigr)^{-1} \, d\mu(x) < \eps \; .$$
Since $\cF_n$ is a right F{\o}lner sequence, we can take $n_0$ such that
$$\cF'_n := \cF_n \cap \bigcap_{s \in \cL} \cF_n s \quad\text{satisfies}\quad \frac{|\cF_n \setminus \cF'_n|}{|\cF_n|} < \eps$$
for all $n \geq n_0$. When $g \in \cF'_n$, we have $\cL \subset \cF_n^{-1} g$. Also, for every $g \in \cF_n$, we have that $e \in \cF_n^{-1} g$ so that
$$\sum_{k \in \cF_n^{-1} g} \frac{d\mu(k \cdot x)}{d\mu(x)} \geq 1$$
for all $x \in X$. Therefore, for all $n \geq n_0$,
\begin{align*}
\sum_{g \in G} \eta_n(g)^2 \, & \, \int_X \Bigl( \sum_{k \in G} \eta_n(gk^{-1}) \, \frac{d \mu(k \cdot x)}{d \mu(x)} \Bigr)^{-1} \; d \mu(x) \\[1ex] & =
\frac{1}{|\cF_n|} \, \sum_{g \in \cF_n} \int_X \Bigl( \sum_{k \in \cF_n^{-1} g} \frac{d\mu(k \cdot x)}{d\mu(x)} \Bigr)^{-1} \; d\mu(x) \\[1ex]
&\leq \frac{|\cF_n \setminus \cF'_n|}{|\cF_n|} + \frac{1}{|\cF_n|} \sum_{g \in \cF'_n} \int_X \Bigl( \sum_{k \in \cL} \frac{d\mu(k \cdot x)}{d\mu(x)} \Bigr)^{-1} \; d\mu(x) < 2\eps \; .
\end{align*}
So the sequence $\eta_n$ is strongly recurrent and $G \actson (X,\mu)$ is strongly conservative.

3. Let $\eta$ be a probability measure on $G$. By convexity of $t \mapsto t^{-1}$, we have
$$\sum_{k \in G} \eta(gk^{-1}) \, \frac{d \mu(k \cdot x)}{d \mu(x)} \geq \Bigl( \sum_{k \in G} \eta(gk^{-1}) \, \frac{d \mu(x)}{d \mu(k \cdot x)} \Bigr)^{-1} \; .$$
It follows that the left hand side of \eqref{eq.strongly-conservative} is bounded above by the expression in \eqref{eq.est-2}.
\end{proof}

\begin{lemma}\label{lem.consequence-strong-conservative}
Let $G$ be a countable group and $G \actson (X,\mu)$ a strongly conservative nonsingular action on a standard probability space $(X,\mu)$. Let $\vphi \in \Aut(X,\mu)$ be a nonsingular automorphism satisfying the following two properties.
\begin{enumlist}
\item There exist a $C > 0$ such that
\begin{equation}\label{eq.bound-C}
C^{-1} \leq \frac{d\mu(g \cdot \vphi(x))}{d\mu(g \cdot x)} \leq C
\end{equation}
for all $g \in G$ and a.e.\ $x \in X$.
\item There is an $L^1$-dense set of functions $F_0 \in L^\infty(X)$ with the property that
\begin{equation}\label{eq.property-F0}
|F_0(g \cdot \vphi(x)) - F_0(g \cdot x)| \recht 0 \quad\text{uniformly in $x \in X$ as $g \recht \infty$.}
\end{equation}
\end{enumlist}
Then every $G$-invariant $F \in L^\infty(X)$ satisfies $F(\vphi(x)) = F(x)$ for a.e.\ $x \in X$.
\end{lemma}
\begin{proof}
Fix a sequence of probability measures $\eta_n$ on $G$ satisfying \eqref{eq.strongly-conservative}. Denote
$$\xi_n : G \times X \recht [0,1] : \xi_n(h,x) = \sum_{g \in G} \eta_n(g) \; \frac{\dis\eta_n(gh^{-1}) \, \frac{d\mu(h \cdot x)}{d \mu(x)}}{\dis\sum_{k \in G} \eta_n(gk^{-1}) \, \frac{d\mu(k \cdot x)}{d \mu(x)}} \; .$$
By definition, $\sum_{h \in G} \xi_n(h,x) = 1$ for all $x \in X$ and $n \in \N$. By Lemma \ref{lem.ucp}, the maps
$$\theta_n : L^\infty(X) \recht L^\infty(X) : (\theta_n(H))(x) = \sum_{h \in G} \xi_n(h,x) \, H(h \cdot x)$$
satisfy $\|\theta_n(H)\|_1 \leq \|H\|_1$ for all $H \in L^\infty(X)$.

We claim that for every fixed $h \in G$, we have $\lim_n \|\xi_n(h,\cdot)\|_1 = 0$. To prove this claim, first note that by changing the variable $x$ to $g^{-1} \cdot x$, we get that
\begin{equation}\label{eq.hulp}
\begin{split}
\|\xi_n(h,\cdot)\|_1 &= \sum_{g \in G} \eta_n(g) \; \int_X \frac{\eta_n(gh^{-1}) \, \frac{d\mu(hg^{-1} \cdot x)}{d\mu(x)}}{\sum_{k \in G} \eta_n(k) \, \frac{d\mu(k^{-1} \cdot x)}{d \mu(x)}} \; \frac{d\mu(g^{-1} \cdot x)}{d\mu(x)} \; d\mu(x) \\[3ex]
&= \int_X \frac{\sum_{g \in G} \eta_n(g) \, \frac{d\mu(g^{-1} \cdot x)}{d\mu(x)} \, \eta_n(gh^{-1}) \, \frac{d\mu(hg^{-1} \cdot x)}{d\mu(x)}}{\sum_{k \in G} \eta_n(k) \, \frac{d\mu(k^{-1} \cdot x)}{d \mu(x)}} \; d\mu(x) \; .
\end{split}
\end{equation}
Using the Cauchy-Schwarz inequality, we conclude that
$$\|\xi_n(h,\cdot)\|_1 \leq \int \frac{\sum_{g \in G} \eta_n(g)^2 \, \Bigl(\frac{d\mu(g^{-1} \cdot x)}{d\mu(x)} \Bigr)^2}{\sum_{k \in G} \eta_n(k) \, \frac{d\mu(k^{-1} \cdot x)}{d \mu(x)}} \; d\mu(x) \; .$$
Applying \eqref{eq.hulp} to $h=e$, it follows that $\|\xi_n(h,\cdot)\|_1 \leq \|\xi_n(e,\cdot)\|_1$. Condition \eqref{eq.strongly-conservative} is precisely saying that $\|\xi_n(e,\cdot)\|_1 \recht 0$. So the claim is proven.

Now assume that $F : X \recht [0,1]$ is $G$-invariant. We have to prove that $F(\vphi(x)) = F(x)$ for a.e.\ $x \in X$. Choose $\eps > 0$. Take a function $F_0 : X \recht [0,1]$ satisfying \eqref{eq.property-F0} with $\|F-F_0\|_1 < \eps$. Then take a finite subset $\cF \subset G$ such that $|F_0(g \cdot \vphi(x))-F_0(g \cdot x)| < \eps$ for all $g \in G \setminus \cF$ and all $x \in X$. Using the claim above, fix $n$ large enough such that
\begin{equation}\label{eq.remainder}
\sum_{h \in \cF} \|\xi_n(h,\cdot)\|_1 < \eps \; .
\end{equation}
Denote
$$\Psi : L^\infty(X) \recht L^\infty(X) : (\Psi(H))(x) = \sum_{h \in G \setminus \cF} \xi_n(h,x) \, H(h \cdot x) \; .$$
It follows from \eqref{eq.remainder} that $\|\Psi(H) - \theta_n(H)\|_1 \leq \eps \, \|H\|_\infty$ for all $H \in L^\infty(X)$.

Since $F = \theta_n(F)$, we get that
$$\|F - \Psi(F_0)\|_1 \leq \|\theta_n(F) - \theta_n(F_0)\|_1 + \|\theta_n(F_0) - \Psi(F_0)\|_1 < 2 \eps \; .$$
Let $C > 0$ be the constant given by \eqref{eq.bound-C}. It follows that
\begin{equation}\label{eq.step1}
\|F \circ \vphi - \Psi(F_0) \circ \vphi\|_1 \leq C \, \|F - \Psi(F_0)\|_1 < 2C \, \eps \; .
\end{equation}
Note that
$$(\Psi(F_0))(\vphi(x)) = \sum_{h \in G \setminus \cF} \xi_n(h,\vphi(x)) \, F_0(h \cdot \vphi(x)) \; .$$
Define the map
$$\Theta : L^\infty(X) \recht L^\infty(X) : (\Theta(H))(x) = \sum_{h \in G \setminus \cF} \xi_n(h,\vphi(x)) \, H(h \cdot x) \; .$$
Since $|F_0(h \cdot \vphi(x)) - F_0(h \cdot x)| \leq \eps$ for all $h \in G \setminus \cF$ and all $x \in X$, we conclude that
\begin{equation}\label{eq.step2}
\|\Psi(F_0) \circ \vphi - \Theta(F_0)\|_1 \leq \eps \; .
\end{equation}
By \eqref{eq.bound-C}, we get that $\xi_n(h,\vphi(x)) \leq C^2 \, \xi_n(h,x)$ for all $h \in G$ and a.e.\ $x \in X$. It follows that for all $H \in L^\infty(X)$, and for a.e.\ $x \in X$,
$$|(\Theta(H))(x)| \leq C^2 \sum_{h \in G \setminus \cF} \xi_n(h,x) \, |H(h \cdot x)| \leq C^2 \, \theta_n(|H|)(x) \; .$$
So, $\|\Theta(H)\|_1 \leq C^2 \, \|H\|_1$ for all $H \in L^\infty(X)$. In particular, $\|\Theta(F_0) - \Theta(F)\|_1 < C^2 \, \eps$. In combination with \eqref{eq.step1} and \eqref{eq.step2}, it follows that $\|F \circ \vphi - \Theta(F)\|_1 < (C+1)^2 \, \eps$.

Since $F$ is $G$-invariant, we also find that
$$(\Theta(F))(x) - F(x) = F(x) \, \sum_{h \in \cF} \xi_n(h,\vphi(x)) \; .$$
Since $\|F\|_\infty \leq 1$, it follows that
$$\|\Theta(F) - F\|_1 \leq \sum_{h \in \cF} \|\xi_n(h,\vphi(\cdot))\|_1 \leq C^2 \sum_{h \in \cF} \|\xi_n(h,\cdot)\|_1 < C^2 \, \eps \; .$$
So we conclude that $\|F \circ \vphi - F\|_1 < 2(C+1)^2 \, \eps$. Since this holds for all $\eps > 0$, we get that $F(\vphi(x)) = F(x)$ for a.e.\ $x \in X$.
\end{proof}

\section{Ergodicity of Bernoulli actions of arbitrary groups}

Let $G$ be a countable group. Fix $0 < \delta < 1/2$ and fix probability measures $\mu_g$ on $\{0,1\}$ satisfying $\delta \leq \mu_g(0) \leq 1-\delta$ for every $g \in G$. Define $c_g(k) = \mu_k(0) - \mu_{g^{-1}k}(0)$. Assume that $c_g \in \ell^2(G)$ for every $g \in G$. Then the Bernoulli action
\begin{equation}\label{eq.Bernoulli}
G \actson (X,\mu) = \prod_{g \in G} (\{0,1\},\mu_g)
\end{equation}
is nonsingular.

We start by proving the second half of Theorem \ref{thm-A}.

\begin{theorem}\label{thm.general-ergodicity}
If the Bernoulli action $G \actson (X,\mu)$ in \eqref{eq.Bernoulli}, with $\mu_g(0) \in [\delta,1-\delta]$ for all $g \in G$, is strongly conservative, then it is weakly mixing.
\end{theorem}

Combining Theorem \ref{thm.general-ergodicity} with Proposition \ref{prop.characterize-strong-conservative}, we get as a corollary another proof of \cite[Theorem 0.2]{D18}.

\begin{corollary}\label{cor.amenable}
If $G$ is amenable and if the Bernoulli action $G \actson (X,\mu)$ in \eqref{eq.Bernoulli}, with $\mu_g(0) \in [\delta,1-\delta]$ for all $g \in G$, is conservative, then it is weakly mixing.
\end{corollary}

\begin{proof}[Proof of Theorem \ref{thm.general-ergodicity}]
Let $G \actson (Y,\eta)$ be any ergodic pmp action. Note that the diagonal action $G \actson (X \times Y,\mu \times \eta)$ stays strongly conservative. Denote by $\tau : \{0,1\} \recht \{0,1\} : \tau(i) = 1-i$ the map that exchanges $0$ and $1$. For every $s \in G$, define $\tau_s \in \Aut(X,\mu)$ given by changing the $s$'th coordinate, i.e.\ $(\tau_s(x))_k = x_k$ if $k \neq s$ and $(\tau_s(x))_s = \tau(x_s)$. When $\cF \subset G$ is a finite subset and $F_0 \in L^\infty(\{0,1\}^\cF \times Y) \subset L^\infty(X \times Y)$, we have that $F_0(g \cdot (\tau_s(x),y)) = F_0(g \cdot (x,y))$ for all $g \in G \setminus \cF s^{-1}$ and all $(x,y) \in X\times Y$. The transformation $\tau_s \times \id$ also satisfies \eqref{eq.bound-C} with $C = (1-\delta)/\delta$.

Let $F \in L^\infty(X \times Y)$ be a $G$-invariant function. It follows from Lemma \ref{lem.consequence-strong-conservative} that $F(\tau_s(x),y) = F(x,y)$ for all $s \in G$ and a.e.\ $(x,y) \in X \times Y$. This means that $F$ only depends on the second variable. But then, $F$ must be constant a.e., by the ergodicity of $G \actson (Y,\eta)$.
\end{proof}

Proposition 4.1 in \cite{VW17} provides a sufficient condition for the conservativeness of a Bernoulli action. We modify that argument to obtain the following criterion for strong conservativeness, thus implying ergodicity and weak mixing by Theorem \ref{thm.general-ergodicity}. For later use, we also prove that the Maharam extension stays strongly conservative.

For every $\al > 0$, we define the probability measure $\nu_\al$ on $\R$ given by
\begin{equation}\label{eq.nu-alpha}
d\nu_\al(t) = \frac{\al}{2} \, \exp(-\al |t|) \, dt \; .
\end{equation}

\begin{proposition}\label{prop.strong-conservative-Bernoulli}
Consider the Bernoulli action $G \actson (X,\mu)$ in \eqref{eq.Bernoulli} with $\mu_g(0) \in [\delta,1-\delta]$ for all $g \in G$. Assume that $\kappa > \delta^{-1} (1-\delta)^{-1}$ and that
$$\sum_{g \in G} \exp(- 4\kappa \, \|c_g\|_2^2) = +\infty \; .$$
Then the Bernoulli action $G \actson (X,\mu)$ is strongly conservative.

More precisely, whenever $\delta^{-1} (1-\delta)^{-1} < \kappa_1 < \kappa$, there exists a sequence of probability measures $\eta_n$ on $G$ satisfying
\begin{equation}\label{eq.eta-n-we-need}
\sum_{k,g \in G} \eta_n(g)^2 \, \eta_n(gk^{-1}) \, \exp\bigl(\kappa_1 \, \|c_k\|_2^2\bigr) \; \recht \; 0 \quad\text{as $n \recht \infty$.}
\end{equation}
Each sequence of probability measures $\eta_n$ satisfying \eqref{eq.eta-n-we-need} is strongly recurrent for $G \actson (X,\mu)$.

Also, for every $\kappa_1$ as above, there exists an $\al_0 > 0$ such that each sequence $\eta_n$ satisfying \eqref{eq.eta-n-we-need} is strongly recurrent for the Maharam extension $G \actson (X \times \R,\mu \times \nu_\al)$ for every $\al \in (0,\al_0)$.
\end{proposition}

Note that when $\delta$ is close to $0$, then the bound in Proposition \ref{prop.strong-conservative-Bernoulli}, guaranteeing strong conservativeness and ergodicity, is sharper than the conservativeness bound in \cite[Proposition 4.1]{VW17}.

\begin{proof}
Fix $\delta^{-1} (1-\delta)^{-1} < \kappa_1 < \kappa$. Choose $\kappa_2$ such that $\kappa_1 < \kappa_2 < \kappa$. By (4.2) in \cite{VW17}, we can fix an increasing sequence $s_n \in (0,+\infty)$ tending to infinity and finite subsets $\cF_n \subset G$ such that for all $n$, we have that
$$|\cF_n| \geq \exp(4\kappa_2 s_n) \quad\text{and}\quad \|c_k\|_2^2 \leq s_n \;\;\text{for all $k \in \cF_n$.}$$
Define $\eta_n$ as the uniform probability measure on $\cF_n$. We prove that the sequence $\eta_n$ satisfies \eqref{eq.eta-n-we-need}.

When $k \in \cF_n^{-1} \cF_n$, writing $k = h^{-1} g$ with $g,h \in \cF_n$, we have $c_k = c_{h^{-1}} + \lambda_h^* c_g$, so that
$$\|c_k\|_2^2 \leq 2(\|c_h\|_2^2 + \|c_g\|_2^2) \leq 4 s_n \; .$$
It then follows that
\begin{align*}
\sum_{k,g \in G} \eta_n(g)^2 \, \eta_n(gk^{-1}) \, \exp\bigl(\kappa_1 \, \|c_k\|_2^2\bigr) &\leq \sum_{k,g \in G} \eta_n(g)^2 \, \eta_n(gk^{-1}) \, \exp(4 \kappa_1 s_n) \\
&= \frac{1}{|\cF_n|} \, \exp(4 \kappa_1 s_n) \leq \exp(4 (\kappa_1 - \kappa_2) s_n) \recht 0 \; .
\end{align*}

Next assume that $\eta_n$ is any sequence of probability measures satisfying \eqref{eq.eta-n-we-need}. We first prove that $\eta_n$ is strongly recurrent for $G \actson (X,\mu)$.

With convergence a.e.\ we have that
\begin{equation}\label{eq.RN}
\frac{d \mu(k \cdot x)}{d\mu(x)} = \prod_{h \in G} \frac{\mu_{kh}(x_h)}{\mu_h(x_h)} \; .
\end{equation}
Write $\kappa_0 = \delta^{-1}(1-\delta)^{-1}$. Whenever $\delta \leq a,b \leq 1-\delta$, we have that
$$\frac{a^2}{b} + \frac{(1-a)^2}{1-b} = 1 + \frac{(a-b)^2}{b(1-b)} \leq 1 + \kappa_0 \, (a-b)^2 \; .$$
Define $\gamma : G \recht [\delta,1-\delta] : \gamma(g) = \mu_g(0)$. Let $G = \{h_1,h_2,\ldots\}$. With a similar computation as in the proof of \cite[Proposition 4.1]{VW17}, by \eqref{eq.RN} and Fatou's lemma, we have
\begin{align}
\int_X \frac{d\mu(x)}{d\mu(k \cdot x)} \, d\mu(x) & \leq \liminf_{n} \prod_{i=1}^n \Bigl( \frac{\gamma(h_i)^2}{\gamma(kh_i)} + \frac{(1-\gamma(h_i))^2}{1-\gamma(kh_i)}\Bigr) \notag\\
& \leq \liminf_{n}  \prod_{i=1}^n \Bigl( 1 + \kappa_0 (\gamma(kh_i) - \gamma(h_i))^2 \Bigr) \notag\\
& \leq \liminf_n \exp\Bigl( \kappa_0 \sum_{i=1}^n (\gamma(kh_i) - \gamma(h_i))^2 \Bigr) = \exp(\kappa_0 \|c_k\|_2^2) \; .\label{eq.estimate-integral}
\end{align}
From \eqref{eq.estimate-integral} and \eqref{eq.eta-n-we-need}, it follows that the sequence $\eta_n$ satisfies \eqref{eq.est-2}. By Proposition \ref{prop.characterize-strong-conservative}, the sequence $\eta_n$ is strongly recurrent.

Finally, having fixed $\kappa_1$ with $\kappa_0 < \kappa_1 < \kappa$, we construct an $\al_0 > 0$ such that any sequence of probability measures satisfying \eqref{eq.eta-n-we-need} is strongly recurrent for the Maharam extension $G \actson (X \times \R,\mu \times \nu_\al)$, for all $\al \in (0,\al_0)$.

First note that
$$\exp(-\al |a|) \leq \frac{d\nu_\al(t+a)}{d\nu_\al(t)} \leq \exp(\al |a|) \quad\text{for all $a,t \in \R$.}$$
Writing $\mu_\al = \mu \times \nu_\al$, it follows that
$$\frac{d\mu_\al(k \cdot (x,t))}{d\mu_\al(x,t)} \quad\text{lies between}\;\; \Bigl(\frac{d \mu(k \cdot x)}{d\mu(x)}\Bigr)^{1 \pm \al} \; ,$$
so that
\begin{equation}\label{eq.this-helps}
\frac{d\mu_\al(x,t)}{d\mu_\al(k \cdot (x,t))} \leq \Bigl(\frac{d\mu(x)}{d \mu(k \cdot x)}\Bigr)^{1 + \al} + \Bigl(\frac{d\mu(x)}{d \mu(k \cdot x)}\Bigr)^{1 - \al}
\end{equation}
for all $(x,t) \in X \times \R$ and $k \in G$.

Using a second order Taylor expansion, we get that for every $\beta > 0$ and $a,b \in [\delta,1-\delta]$, there exist $c$ and $d$ lying between $a$ and $b$ such that
\begin{align*}
\frac{a^{\beta + 1}}{b^\beta} + \frac{(1-a)^{\beta +1}}{(1-b)^\beta} = \; & 1 + (b-a)^2 \, \eta_\beta(b,c,d) \\ &\text{with}\;\;\eta_\beta(b,c,d)= \frac{\beta (\beta +1)}{2} \, \frac{c^{\beta -1} (1-b)^\beta + (1-d)^{\beta-1} b^\beta}{b^\beta \, (1-b)^\beta} \; .
\end{align*}
Defining
$$\eta_\beta = \max \{\eta_\beta(b,c,d) \mid b,c,d \in [\delta,1-\delta]\} \, $$
we conclude that
\begin{equation}\label{eq.this-also-helps}
\frac{a^{\beta + 1}}{b^\beta} + \frac{(1-a)^{\beta +1}}{(1-b)^\beta} \leq 1 + \eta_\beta \, (b-a)^2 \quad\text{for all}\;\; a,b \in [\delta,1-\delta]
\end{equation}
and $\eta_\beta \recht \kappa_0 = \delta^{-1}(1-\delta)^{-1}$ when $\beta \recht 1$.

Take $\al_0>0$ small enough such that $\eta_{1 \pm \al} \leq \kappa_1$ for all $\al \in (0,\al_0)$. Combining \eqref{eq.this-helps} and \eqref{eq.this-also-helps}, and making a similar computation as in \eqref{eq.estimate-integral}, we find that
$$\int_X \frac{d\mu_\al(x,t)}{d\mu_\al(k \cdot (x,t))} \, d\mu_\al(x,t) \leq 2\exp(\kappa_1 \, \|c_k\|_2^2)$$
for all $\al \in (0,\al_0)$ and all $k \in G$. In combination with \eqref{eq.eta-n-we-need}, it again follows that the sequence $\eta_n$ satisfies \eqref{eq.est-2}, so that $\eta_n$ is strongly recurrent for the Maharam extension $G \actson (X\times \R,\mu_\al)$.
\end{proof}

\section{The type of Bernoulli actions of arbitrary groups}

Recall that a subset $W \subset G$ of a group $G$ is called \emph{almost invariant} if $|gW \vartriangle W| < \infty$ for every $g \in G$. A group $G$ is said to have \emph{more than one end} if $G$ admits an almost invariant subset $W \subset G$ such that both $W$ and $G \setminus W$ are infinite. By the version of Stallings' theorem for possibly infinitely generated groups (see \cite[Theorem IV.6.10]{DD89} and the discussion in Remark \ref{rem.ends-groups}), the groups with more than one end can be exactly described.

\begin{theorem}\label{thm.main-general}
Let $G$ be a countable infinite group. If $H^1(G,\ell^2(G)) \neq \{0\}$, then $G$ admits a nonsingular Bernoulli action of stable type III$_1$.

If $G$ has more than one end and $G$ is not virtually cyclic, then $G$ admits nonsingular Bernoulli actions of stable type III$_\lambda$ for each $\lambda \in (0,1]$ close enough to $1$.
\end{theorem}

We prove Theorem \ref{thm.main-general} below (see Theorem \ref{thm.other-version-main-thm} and the discussion preceding it). By \cite[Theorem 3.1]{VW17}, also the converse of the first statement holds. So we immediately get the following corollary, stated as Theorem \ref{thm-D} above.

\begin{corollary}\label{cor.solution-VW-conjecture}
A countable infinite group $G$ admits a nonsingular Bernoulli action of type III$_1$ if and only if $H^1(G,\ell^2(G)) \neq \{0\}$.
\end{corollary}

For nonsingular Bernoulli actions satisfying the conservativeness criterion in Proposition \ref{prop.strong-conservative-Bernoulli}, by Theorem \ref{thm.type-bernoulli} below, also the converse of the second statement in Theorem \ref{thm.main-general} holds: if such a Bernoulli action is of type III$_\lambda$ for some $\lambda \in (0,1)$, then $G$ must have more than one end.

We derive Theorem \ref{thm.main-general} from a general result determining the (stable) type of an arbitrary nonsingular Bernoulli action $G \actson (X,\mu)$, provided that $\mu_g(0) \in [\delta,1-\delta]$ for all $g \in G$ and provided that the conservativeness criterion in Proposition \ref{prop.strong-conservative-Bernoulli} holds.

When $W \subset G$ is an almost invariant subset, $c_W : g \mapsto 1_W - 1_{gW}$ is a $1$-cocycle with values in $\ell^1(G) \subset \ell^2(G)$. The cocycle $c_W$ is a coboundary if and only if either $W$ or its complement $G \setminus W$ is finite. Denote by $\Zai(G,\ell^2(G)) \subset Z^1(G,\ell^2(G))$ the linear span of all $1$-cocycles $c_W$ associated with almost invariant subsets $W \subset G$. Under the above hypotheses, we prove in the following theorem that $G \actson (X,\mu)$ is of stable type III$_1$, unless the associated $1$-cocycle $c$ is inner (in which case we obviously get a measure preserving Bernoulli action) or the group $G$ has more than one end and $c$ is cohomologous to a $1$-cocycle in $\Zai(G,\ell^2(G))$ (in which case, type III$_\lambda$ with $\lambda \in (0,1)$ is possible).

\begin{theorem}\label{thm.type-bernoulli}
Let $G \actson (X,\mu) = \prod_{g \in G} (\{0,1\},\mu_g)$ be a nonsingular Bernoulli action with $\mu_g(0) \in [\delta,1-\delta]$ for all $g \in G$ and with $1$-cocycle $c_g(k) = \mu_k(0) -\mu_{g^{-1}k}(0)$. Assume that at least one of the following conditions hold.
\begin{itemlist}
\item $G$ is amenable and $G \actson (X,\mu)$ is conservative.
\item We have $\dis\sum_{g \in G} \exp(- 4\kappa \, \|c_g\|_2^2) = +\infty \quad\text{for some $\kappa > \delta^{-1} (1-\delta)^{-1}$.}$
\end{itemlist}
Then the following holds.
\begin{enumlist}
\item If $c$ is a coboundary, then $G \actson (X,\mu)$ is of stable type II$_1$ and $\mu \sim \nu^G$ for some probability measure $\nu$ on $\{0,1\}$.

\item If $c$ is not a coboundary, but cohomologous to a $1$-cocycle in $\Zai(G,\ell^2(G))$, then $G \actson (X,\mu)$ is of type III$_\lambda$ for some $\lambda \in (0,1]$ and the precise (stable) type is given in Remark \ref{rem.stable-type}.

\item If $c$ is not cohomologous to a $1$-cocycle in $\Zai(G,\ell^2(G))$, then $G \actson (X,\mu)$ is of stable type III$_1$.
\end{enumlist}
\end{theorem}

Note that Theorem \ref{thm-C} stated in the introduction is a special case of Theorem \ref{thm.type-bernoulli}.

\begin{remark}\label{rem.stable-type}
Write $\gamma(g) = \mu_g(0)$. The cocycle $c$ in Theorem \ref{thm.type-bernoulli} not being a coboundary, but cohomologous to a $1$-cocycle in $\Zai(G,\ell^2(G))$, is equivalent to the existence of a partition $G = W_1 \sqcup \cdots \sqcup W_n$ of $G$ into $2 \leq n < +\infty$ disjoint almost invariant infinite subsets and the existence of a function $\zeta : G \recht [\delta,1-\delta]$ taking (distinct) constant values $\lambda_i$ on each $W_i$, such that $\gamma-\zeta \in \ell^2(G)$. The type and stable type of $G \actson (X,\mu)$ are then determined as follows in terms of $\zeta$.

When $W \subset G$ is almost invariant, composing the $1$-cocycle $c_W$ with the sum $\ell^1(G) \recht \C$, we get that
\begin{equation}\label{eq.group-hom-Om-W}
\Om_W : G \recht \Z : \Om_W(g) = |W \setminus g W| - |g W \setminus W|
\end{equation}
is a group homomorphism.

Define the subgroup $\Lambda \subset \R$ generated by
\begin{equation}\label{eq.group-Lambda}
\log\Bigl(\frac{\lambda_i}{1-\lambda_i}\Bigr) - \log\Bigl(\frac{\lambda_j}{1-\lambda_j}\Bigr) \;\;\text{for all $i \neq j$,}
\end{equation}
and define the group homomorphism
\begin{equation}\label{eq.group-hom-Om}
\Om : G \recht \R : \Om(g) = \sum_{i=1}^n \Om_{W_i}(g) \, \log(\lambda_i) \; .
\end{equation}
\begin{enumlist}
\item If $\Lambda \subset \R$ is dense, then $G \actson (X,\mu)$ is of stable type III$_1$. If $\Lambda = a \Z$ with $a > 0$ and if $\Om(G) \subset a \Z$, then $G \actson (X,\mu)$ is of stable type III$_\lambda$ with $\lambda = \exp(-a)$.

\item If $\Lambda = a \Z$ with $a > 0$ and if $\Om(G) + a \Z \subset \R$ is dense, then $G \actson (X,\mu)$ is of type III$_1$, but not of stable type III$_1$. The possible types of $G \actson X \times Y$ for $G \actson (Y,\eta)$ pmp ergodic, range over III$_1$ and III$_\lambda$ with $\lambda = \exp(-a/k)$ and $k \in \N$.

\item If $\Lambda = a \Z$ with $a > 0$ and if $\Om(G) + a \Z = (a/k_0) \Z$ for $k_0 \in \{2,3,\ldots\}$, then $G \actson (X,\mu)$ is of type III$_\lambda$ with $\lambda = \exp(-(a/k_0))$, but not of stable type III$_\lambda$. The possible types of $G \actson X \times Y$ for $G \actson (Y,\eta)$ pmp ergodic, range over III$_\lambda$ with $\lambda = \exp(-a/k)$ and $k \in \N$ dividing $k_0$.
\end{enumlist}
\end{remark}

As for the proof of Theorems \ref{thm.type-abelian} and \ref{thm.type-abelian-locally-finite}, the main technical step is to prove that functions that are invariant under the Maharam extension of $G \actson (X,\mu)$ are automatically invariant under the Maharam extension of the permutation action $\cS_G \actson (X,\mu)$. So we prove the following variant of Lemma \ref{lem.automatic-invariance-permutation-group}. The proof of the lemma is a substantial refinement of the proofs of Lemma \ref{lem.automatic-invariance-permutation-group} and Theorem \ref{thm.general-ergodicity}. Given a $G$-invariant function $F \in L^\infty(X \times \R)$, we study the behavior of $F(\tau_s(x),t)$, where $\tau_s \in \Aut(X,\mu)$ changes the $s$'th coordinate. Compared to the proof of Lemma \ref{lem.automatic-invariance-permutation-group}, two complications arise. First, it need no longer be true that $\mu_{ga}(0)/\mu_{gb}(0)$ essentially converges to $1$ as $g \recht \infty$ with $a,b$ fixed. Secondly, as $G$ can be nonamenable, we cannot apply the ergodic theorem of \cite[Theorem A.1]{D18}, but we have to use Lemma \ref{lem.ucp} instead.

\begin{lemma}\label{lem.general-automatic-invariance-permutation-group}
Assume that the hypotheses of Theorem \ref{thm.type-bernoulli} hold. If $G \actson (Y,\eta)$ is any ergodic pmp action and $G \actson X \times \R \times Y$ is its diagonal product with the Maharam extension of $G \actson (X,\mu)$, then any $G$-invariant function $F \in L^\infty(X \times \R \times Y)$ satisfies
\begin{equation}\label{eq.again-extra-invariance}
F(\si \cdot x , t + \be(\si,x), y) = F(x,t,y) \quad\text{for all $\si \in \cS_G$ and a.e.\ $(x,t,y) \in X \times \R \times Y$,}
\end{equation}
where $\be(\si,x) = \log(d\mu(\si \cdot x)/d\mu(x))$.
\end{lemma}

Before proving Lemma \ref{lem.general-automatic-invariance-permutation-group}, we prove the following result about arbitrary ergodic nonsingular actions. Given $H \in L^\infty(\R)$, we define
$$\per(H) = \{p \in \R \mid H(t + p) = H(t) \;\;\text{for a.e.\ $t \in \R$}\;\} \; .$$
Since the translation action of $\R$ on $L^\infty(\R)$ is weak$^*$ continuous, $\per(H)$ is always a closed subgroup of $\R$.

\begin{lemma}\label{lem.technical-periods}
Let $G$ be any countable group and $G \actson (Z,\zeta)$ any nonsingular ergodic action with Maharam extension $G \actson Z \times \R$. Let $F : Z \times \R \recht \C$ be a bounded $G$-invariant Borel function. Then exactly one of the following statements holds.
\begin{enumlist}
\item The function $F$ is essentially constant.
\item For a.e.\ $z \in Z$, we have that $\per(F(z,\cdot)) = \{0\}$.
\item There exists a $p > 0$ such that $\per(F(z,\cdot)) = p \Z$ for a.e.\ $z \in Z$.
\end{enumlist}
\end{lemma}
\begin{proof}
Since $F$ induces the Borel map $z \mapsto F(z,\cdot)$ from $Z$ to $L^\infty(\R)$ equipped with the weak$^*$ topology, the set of $z \in Z$ such that $F(z,\cdot)$ is essentially constant, is Borel. By the $G$-invariance of $F$, this set is $G$-invariant. So either 1 holds, or we find a $G$-invariant Borel set $\cU \subset Z$ with $\zeta(\cU) = 1$ such that for every $z \in \cU$, the function $F(z,\cdot)$ is not essentially constant. In the latter case, we define the Borel set
$$\cE = \{(z,p) \in \cU \times \R \mid F(z,t) = F(z,t+p) \;\;\text{for a.e.\ $t \in \R$}\;\} \; .$$
So $(z,p) \in \cE$ if and only if $p \in \per(F(z,\cdot))$. For $z \in \cU$, the function $F(z,\cdot)$ is not essentially constant, so that $\per(F(z,\cdot)) \neq \R$. It follows that for every $z \in \cU$, there is a unique $p(z) \geq 0$ such that $\per(F(z,\cdot)) = p(z) \Z$. By the $G$-invariance of $F$, the map $z \mapsto p(z)$ is $G$-invariant. Since
the map $(z,p) \mapsto z$ from $\cE$ to $\cU$ is countable-to-one, $p$ is Borel.

So either $p(z) = 0$ for a.e.\ $z \in \cU$ and 2 holds, or $p(z) = p > 0$ for a.e.\ $z \in \cU$ and 3 holds.
\end{proof}

We are now ready to prove Lemma \ref{lem.general-automatic-invariance-permutation-group}.

\begin{proof}[Proof of Lemma \ref{lem.general-automatic-invariance-permutation-group}]
Let $G \actson (Y,\eta)$ be an ergodic pmp action and $F \in X \times \R \times Y \recht \C$ a $G$-invariant Borel function. Without lack of generality, we may assume that $F$ takes values in the interval $[0,1]$. By Theorem \ref{thm.general-ergodicity} and Propositions \ref{prop.characterize-strong-conservative} and \ref{prop.strong-conservative-Bernoulli}, the action $G \actson X \times Y$ is ergodic. Applying Lemma \ref{lem.technical-periods} to $G \actson X \times Y$, we are in precisely one of the following cases.

{\bf Case 1.} The function $F$ is essentially constant.

{\bf Case 2.} For a.e.\ $(x,y) \in X \times Y$, we have that $\per(F(x,\cdot,y)) = \{0\}$.

{\bf Case 3.} There exists a $p > 0$ such that $\per(F(x,\cdot,y)) = p \Z$ for a.e.\ $(x,y) \in X \times Y$.

In case 1, \eqref{eq.again-extra-invariance} holds trivially.

{\bf Proof in case 2.} As in the proof of Theorem \ref{thm.general-ergodicity}, denote by $\tau : \{0,1\} \recht \{0,1\} : \tau(i) = 1-i$ the map that exchanges $0$ and $1$, and define, for every $g \in G$, the transformation $\tau_g \in \Aut(X,\mu)$ given by changing the $g$'th coordinate. Denote
$$R_g = \log \frac{\mu_g(0)}{\mu_g(1)} \;\;\text{for all $g \in G$.}$$
Writing $C = \log((1-\delta)/\delta)$, we have $R_g \in [-C,C]$ for all $g \in G$.

{\bf Notation.} To every $R \in \R$, we associate the function $R : \{0,1\} \recht \R$ given by $R(0) = R$ and $R(1) = -R$.

To prove the lemma in case 2, it suffices to prove the following statement.

{\bf Claim 1.} There exists a $\rho \in [-C,C]$ such that
$$F(\tau_s(x),t,y) = F(x,t+R_s(x_s) - \rho(x_s),y) \quad\text{for all $s \in G$ and a.e.\ $(x,t,y) \in X \times \R \times Y$.}$$
Denote by $\nu_\al$ the probability measure on $\R$ given by $d\nu_\al(t) = (\al/2) \exp(-\al |t|) dt$. Write $\mu_\al = \mu \times \nu_\al$. Fix $\al_0 > 0$ such that there exists a sequence of probability measures $\eta_n$ on $G$ that is strongly recurrent for $G \actson (X \times \R,\mu_\al)$ for each $\al \in (0,\al_0)$. By the hypotheses of the theorem, Proposition \ref{prop.characterize-strong-conservative} and Proposition \ref{prop.strong-conservative-Bernoulli}, such an $\al_0 > 0$ exists.

In view of applying Lemma \ref{lem.ucp}, define
$$p(\al,n ; k,x,t) = \sum_{g \in G} \eta_n(g) \, \frac{\dis \eta_n(gk^{-1}) \, \frac{d\mu_\al(k\cdot(x,t))}{d\mu_\al(x,t)}}{\dis \sum_{h \in G} \eta_n(gh^{-1}) \, \frac{d\mu_\al(h\cdot(x,t))}{d\mu_\al(x,t)}} \; .$$
Note that $\sum_{k \in G} p(\al,n;k,x,t) = 1$ for all $\al,n,x,t$. Whenever $\al_i \recht 0$ in $(0,\al_0)$ and $n_i \recht +\infty$ in $\N$, write $p_i(k,x) = p(\al_i,n_i; k,x,0)$ and consider the probability measures on $[-2C,2C]$ given by
$$\rho_i(s,x) = \sum_{k \in G} p_i(k,\tau_s(x)) \, \delta(R_s(x_s)-R_{ks}(x_s)) \; ,$$
where $\delta(R)$ denotes the Dirac measure in $R \in \R$. The main step towards proving claim 1, is to prove the following statement.

{\bf Claim 2.} There exist sequences $\al_i \recht 0$ in $(0,\al_0)$ and $n_i \recht +\infty$ in $\N$ such that for all $s \in G$ and a.e.\ $x \in X$, the probability measures $\rho_i(s,x)$ converge weakly$^*$ to a Dirac measure $\delta(\rho(s,x))$ and
$$F(\tau_s(x),t,y) = F(x,t+\rho(s,x),y) \quad\text{for all $s \in G$ and a.e.\ $(x,t,y) \in X \times \R \times Y$.}$$

Write
$$F_\al(t,s) = \frac{d\nu_\al(t+s)}{d\nu_\al(t)} = \exp(-\al |t+s| + \al |t|) \; .$$
Note that
\begin{equation}\label{eq.approx-inv}
\exp(-\al |s|) \leq F_\al(t,s) \leq \exp(\al |s|) \quad\text{for all $s,t \in \R$.}
\end{equation}
The Maharam extension $G \actson X \times \R$ is given by $g \cdot (x,t) = (g \cdot x, t + \al(g,x))$. Then,
$$\frac{d\mu_\al(k\cdot (x,t))}{d\mu_\al(x,t)} = \frac{d\mu(k \cdot x)}{d\mu(x)} \, F_\al(t,\al(k,x)) \; .$$
Note that for all $s,t \in \R$,
$$\exp(-2 \al |t-t'|) \leq \frac{F_\al(t,s)}{F_\al(t',s)} \leq \exp(2\al |t-t'|) \; .$$
It follows that
\begin{equation}\label{eq.imp-estim}
\exp(-4\al |t-t'|) \; p(\al,n;k,x,t') \leq p(\al,n;k,x,t) \leq \exp(4\al |t-t'|) \; p(\al,n;k,x,t') \; .
\end{equation}
We continuously use \eqref{eq.imp-estim} to control the asymptotic independence of $p(\al,n;k,x,t)$ of the variable $t$, as $\al \recht 0$.

Also note that
\begin{equation}\label{eq.essential-algebra}
\al(k,\tau_s(x)) = \al(k,x) - R_{ks}(x_s) + R_s(x_s) \; .
\end{equation}
We then find a constant $C_1 \geq 1$ such that
\begin{equation}\label{eq.second-imp-estim}
C_1^{-1} \, p(\al,n;k,x,t) \leq p(\al,n;k,\tau_s(x),t) \leq C_1 \, p(\al,n;k,x,t) \quad\text{for all $\al,n,k,s,x,t$.}
\end{equation}

We have now introduced enough notation to prove claim 2. We denote by $\|\,\cdot\,\|_{\mu_\al}$ the $L^1$-norm w.r.t.\ the probability measure $\mu_\al \times \eta$ on $X \times \R \times Y$. By Lemma \ref{lem.ucp}, the maps
\begin{multline*}
 \Psi_{\al,n} : L^\infty(X \times \R \times Y) \recht L^\infty(X \times \R \times Y) : \\  (\Psi_{\al,n}(H))(x,t,y) = \sum_{k \in G} p(\al,n;k,x,t) \, H(k\cdot (x,t,y))
\end{multline*}
satisfy $\|\Psi_{\al,n}(H)\|_{\mu_\al} \leq \|H\|_{\mu_\al}$ for all $H \in L^\infty(X \times \R \times Y)$.

Also define, for arbitrary $s \in G$, the positive maps
\begin{multline*}
\Theta_{\al,n,s} : L^\infty(X \times \R \times Y) \recht L^\infty(X \times \R \times Y) : \\ (\Theta_{\al,n,s}(H))(x,t,y) = \sum_{k \in G} p(\al,n;k,\tau_s(x),t) \, H(k\cdot x,t+\al(k,\tau_s(x)),k \cdot y) \; .
\end{multline*}
Using \eqref{eq.essential-algebra} and \eqref{eq.second-imp-estim}, every positive $H \in L^\infty(X \times \R \times Y)$ satisfies
\begin{multline*}
\|\Theta_{\al,n,s}(H)\|_{\mu_\al}\\ \leq C_1 \, \sum_{k \in G} \int_{X \times \R \times Y} \; p(\al,n;k,x,t) \; H(k \cdot x, t + \al(k,x) - R_{ks}(x_s) + R_s(x_s),k \cdot y) \; d(\mu_\al \times \eta)(x,t,y) \; .
\end{multline*}
Since $|R_{ks} - R_s| \leq 2 C$, using \eqref{eq.approx-inv} and \eqref{eq.imp-estim}, we then find a constant $C_2>0$ such that for every positive $H \in L^\infty(X \times \R \times Y)$,
\begin{align*}
\|\Theta_{\al,n,s}(H)\|_{\mu_\al} &\leq C_2 \, \sum_{k \in G} \int_{X \times \R \times Y} \; p(\al,n;k,x,t) \; H(k \cdot x, t + \al(k,x),k \cdot y) \; d(\mu_\al \times \eta)(x,t,y) \\ &= C_2 \, \|\Psi_{\al,n}(H)\|_{\mu_\al} \leq C_2 \, \|H\|_{\mu_\al} \; .
\end{align*}

When $\cF \subset G$ is a finite subset and $H \in L^\infty(\{0,1\}^\cF \times \R \times Y)$, we get that
$$H(k\cdot(\tau_s(x),t,y)) = H(k \cdot x,t+\al(k,\tau_s(x)),k \cdot y)$$
for all $k \in G \setminus \cF s^{-1}$.

As in the proof of Lemma \ref{lem.consequence-strong-conservative}, since $\eta_n$ is strongly recurrent for $G \actson (X\times \R,\mu_\al)$, we have for every fixed $k \in G$ and $\al \in (0,\al_0)$ that
$$\lim_{n \recht +\infty} \int_{X \times \R} p(\al,n;k,x,t) \, d\mu_\al(x,t) = 0 \; .$$
Then, using \eqref{eq.second-imp-estim}, we also have that
$$\lim_{n \recht +\infty} \int_{X \times \R} p(\al,n;k,\tau_s(x),t) \, d\mu_\al(x,t) = 0 \; ,$$
for every $s \in G$.

Therefore, as $n \recht +\infty$, in the definition of $\Psi_{\al,n}$ and $\Theta_{\al,n,s}$, we may ignore the terms in the sum with $k \in \cF s^{-1}$. We then conclude that
\begin{equation}\label{eq.conv}
\lim_{n \recht +\infty} \|\Psi_{\al,n}(H) \circ (\tau_s \times \id \times \id) - \Theta_{\al,n,s}(H)\|_{\mu_\al} = 0 \quad\text{for all $\al \in (0,\al_0)$.}
\end{equation}
By the $\|\,\cdot\,\|_{\mu_\al}$-boundedness of $\Psi_{\al,n}$ and $\Theta_{\al,n,s}$, it then follows that \eqref{eq.conv} holds for all $H \in L^\infty(X \times \R \times Y)$. In particular, \eqref{eq.conv} holds for $H = F$.

Finally, define
\begin{multline*}
\Gamma_{\al,n,s} : L^\infty(X \times \R \times Y) \recht L^\infty(X \times \R \times Y) : \\ (\Gamma_{\al,n,s}(H))(x,t,y) = \sum_{k \in G} p(\al,n;k,\tau_s(x),t) \, H(x,t+R_s(x_s)-R_{ks}(x_s), y) \; .
\end{multline*}
By the $G$-invariance of $F$, we have $\Psi_{\al,n}(F) = F$ and $\Theta_{\al,n,s}(F) = \Gamma_{\al,n,s}(F)$. So, we conclude that for every $\al \in (0,\al_0)$ and every $s \in G$,
$$\lim_{n \recht +\infty} \|F \circ (\tau_s \times \id \times \id) - \Gamma_{\al,n,s}(F)\|_{\mu_\al} = 0 \; .$$
Fix any sequence $\al_i \recht 0$ in $(0,\al_0)$. We can then pick $n_i \recht +\infty$, such that
$$\Gamma_{\al_i,n_i,s}(F) \recht F \circ (\tau_s \times \id \times \id) \quad \text{almost everywhere, for all $s \in G$.}$$
Using \eqref{eq.imp-estim} and the fact that $\al_i \recht 0$, it follows that for a.e.\ $(x,y) \in X \times Y$,
\begin{equation}\label{eq.my-step}
\rho_i(s,x) * F(x,\cdot,y) \recht F(\tau_s(x), \cdot,y) \quad\text{a.e.}
\end{equation}
Fix $s \in G$. Fix $(x,y) \in X \times Y$ such that \eqref{eq.my-step} holds for both $(x,y)$ and $(\tau_s(x),y)$, and such that $\per(F(x,\cdot,y)) = \{0\}$. Note that a.e.\ $(x,y) \in X \times Y$ has these properties. Let $\rho_0$ be any weak$^*$ limit point of the sequence of probability measures $\rho_i(s,x)$. We can then take a subsequence $i_j$ such that $\rho_{i_j}(s,x) \recht \rho_0$ weakly$^*$ and also $\rho_{i_j}(s,\tau_s(x)) \recht \rho_1$ weakly$^*$, for some probability measure $\rho_1$. For every $H \in L^\infty(\R)$, the map $\rho \mapsto \rho * H$ is continuous from the space of probability measures on $[-2C,2C]$ equipped with the weak$^*$ topology to $L^\infty(\R)$ equipped with the weak$^*$ topology. By \eqref{eq.my-step}, we get that
$$\rho_0 * F(x,\cdot,y) = F(\tau_s(x),\cdot,y) \quad\text{and}\quad \rho_1 * F(\tau_s(x),\cdot,y) = F(x,\cdot,y) \; .$$
But then $(\rho_1 * \rho_0) * F(x,\cdot,y) = F(x,\cdot,y)$ a.e. Since $\per(F(x,\cdot,y)) = \{0\}$, it follows from the Choquet-Deny theorem (see \cite[Th\'{e}or\`{e}me 1]{CD60}) that $\rho_1 * \rho_0$ is the Dirac measure in $0$. We conclude that $\rho_0$ is a Dirac measure in some point $a \in \R$. By \eqref{eq.my-step}, $F(\tau_s(x),t,y) = F(x,t+a,y)$ for a.e.\ $t \in \R$. Since $\per(F(x,\cdot,y)) = \{0\}$, there is at most one $a \in \R$ satisfying this formula. So we have proved that each weak$^*$ limit point of $\rho_i(s,x)$ is the Dirac measure in the same point. Thus, claim 2 is proven.

Define the probability measures $\zeta_i(s,x)$ on $[-C,C]$ given by
$$\zeta_i(s,x) = \sum_{k \in G} p_i(k,\tau_s(x)) \, \delta(R_{ks}) \; .$$
From claim 2, we get that for all $s \in G$ and a.e.\ $x \in X$, the probability measures $\zeta_i(s,x)$ converge weakly$^*$ to a Dirac measure that we denote by $\delta(\zeta(s,x))$. By \eqref{eq.second-imp-estim}, we have for all $k,g \in G$ and $x \in X$, that $p_i(k,\tau_g(x)) \leq C_1 \, p_i(k,x)$. It then follows that
$$\delta(\zeta(s,\tau_g(x))) \leq C_1 \, \delta(\zeta(s,x)) \; .$$
We conclude that $\zeta(s,\tau_g(x)) = \zeta(s,x)$ for all $g \in G$. This means that $\zeta(s,x)$ is essentially independent of the $x$ variable. We thus define $\zeta_s \in [-C,C]$ such that $\zeta(s,x) = \zeta_s$ for all $s \in G$ and a.e.\ $x \in X$.

To prove claim 1, we have to show that $\zeta_s$ does not depend on $s \in G$. Note that $\zeta_s$ is the unique element of $\R$ satisfying $F(\tau_s(x),t,y) = F(x,t+R_s(x_s)-\zeta_s(x_s),y)$ for a.e.\ $x,t,y$. The uniqueness follows, because $t \mapsto F(x,t,y)$ is not periodic. Moreover, by claim 2, whenever $\eta_n$ is a sequence of probability measures that is strongly recurrent for $G \actson (X \times \R,\mu_\al)$ for all $\al \in (0,\al_0)$, we can choose $\al_i \recht 0$ and $n_i \recht +\infty$ such that $\zeta_i(s,x) \recht \delta(\zeta_s)$ weakly$^*$, for all $s \in G$ and a.e.\ $x \in X$.

{\bf Claim 3.} For every fixed $s \in G$, there exists a sequence of probability measures $\eta_n$ on $G$ that is strongly recurrent for $G \actson (X \times \R,\mu_\al)$ for all $\al \in (0,\al_0)$ and that moreover satisfies $\|s \cdot \eta_n - \eta_n\|_1 \leq 2/3$ for every $n$.

When $G$ is amenable, claim 3 follows immediately from Proposition \ref{prop.characterize-strong-conservative}. Under the second hypothesis, fix $\kappa_1$ with $\delta^{-1} (1-\delta)^{-1} < \kappa_1 < \kappa$ and fix $\al_0 > 0$ so that the conclusion of Proposition \ref{prop.strong-conservative-Bernoulli} holds. Fix $\kappa_2$ such that $\kappa_1 < \kappa_2 < \kappa$. By Proposition \ref{prop.strong-conservative-Bernoulli}, take a sequence of probability measures $\eta_n$ on $G$ such that
$$\sum_{k,g \in G} \eta_n(g)^2 \, \eta_n(gk^{-1}) \, \exp\bigl(\kappa_2 \, \|c_k\|_2^2\bigr) \; \recht \; 0 \; .$$
For every $h \in G$, the probability measure $h \cdot \eta_n$ satisfies
$$\sum_{k,g \in G} (h\cdot\eta_n)(g)^2 \, (h\cdot\eta_n)(gk^{-1}) \, \exp\bigl(\kappa_1 \, \|c_k\|_2^2\bigr) =
\sum_{k,g \in G} \eta_n(g)^2 \, \eta_n(gk^{-1}) \, \exp\bigl(\kappa_1 \, \|c_{h^{-1}kh}\|_2^2\bigr) \; .$$
Since $\|c_{h^{-1} k h}\|_2 \leq \|c_k\|_2 + 2 \|c_h\|_2$ and $\kappa_1 < \kappa_2$, it follows that
$$\sum_{k,g \in G} (h\cdot\eta_n)(g)^2 \, (h\cdot\eta_n)(gk^{-1}) \, \exp\bigl(\kappa_1 \, \|c_k\|_2^2\bigr) \; \recht 0 \; ,$$
for every $h \in G$. We apply this to $h = s$ and $h = s^2$. By Proposition \ref{prop.strong-conservative-Bernoulli}, we conclude that the three sequences $\eta_n$, $s \cdot \eta_n$ and $s^2 \cdot \eta_n$ are strongly recurrent for $G \actson (X \times \R,\mu_\al)$ for all $\al \in (0,\al_0)$. It then follows that also $\eta'_n = (\eta_n + s \cdot \eta_n + s^2 \cdot \eta_n)/3$ satisfies \eqref{eq.eta-n-we-need} and is thus strongly recurrent for $G \actson (X \times \R,\mu_\al)$ for all $\al \in (0,\al_0)$. By construction, $\|s \cdot \eta'_n - \eta'_n\|_1 \leq 2/3$ for every $n$, so that claim 3 is proven.

Fix $s \in G$ and take $\eta_n$ as in claim 3. We now prove that $\zeta_s = \zeta_e$.

Note that
$$p(\al,n;ks^{-1},s \cdot (x,t)) = \sum_{g \in G} (s \cdot \eta_n)(g) \, \frac{\dis \eta_n(gk^{-1}) \, \frac{d\mu_\al(k\cdot(x,t))}{d\mu_\al(x,t)}}{\dis \sum_{h \in G} \eta_n(gh^{-1}) \, \frac{d\mu_\al(h\cdot(x,t))}{d\mu_\al(x,t)}} \; .$$
Using \eqref{eq.imp-estim}, it follows that for a.e.\ $x \in X$,
$$\limsup_{i \recht +\infty} \sum_{k \in G} |p_i(ks^{-1},s \cdot x) - p_i(k,x)| \leq \limsup_{n \recht +\infty} \|s \cdot \eta_n - \eta_n\|_1 \leq \frac{2}{3} \; .$$
But,
$$\sum_{k \in G} p_i(ks^{-1},s \cdot x) \, \delta(R_k) \recht \delta(\zeta_s) \quad\text{and}\quad \sum_{k \in G} p_i(k,x) \, \delta(R_k) \recht \delta(\zeta_e)$$
weakly$^*$, for a.e.\ $x \in X$. Thus, $\|\delta(\zeta_s) - \delta(\zeta_e)\| \leq 2/3$, so that $\zeta_s = \zeta_e$.

Since this holds for all $s \in G$, we find $\rho \in [-C,C]$ such that $\zeta_s = \rho$ for all $s \in G$. So, claim 1 is proven. This concludes the proof of the lemma in case 2.

{\bf Proof in case 3.} Viewing $F$ as a $G$-invariant function for $G \actson X \times \R/p\Z \times Y$ and using that the functions $F(x,\cdot,y)$ have no other periodicity than given by $p \Z$, the proof in case 3 is identical to the proof in case 2.
\end{proof}

\begin{proof}[Proof of Theorem \ref{thm.type-bernoulli} and Remark \ref{rem.stable-type}]
Fix an ergodic pmp action $G \actson (Y,\eta)$. Let $F \in L^\infty(X \times \R \times Y)$ be a $G$-invariant function that generates the fixed point algebra $L^\infty(X \times \R \times Y)^G$. By Lemma \ref{lem.general-automatic-invariance-permutation-group}, the function $F$ is invariant under the Maharam extension of $\cS_G \actson (X,\mu)$.

Denote by $\cL \subset [\delta,1-\delta]$ the closed set of limit values of $\mu_g(0)$ as $g \recht \infty$. If $\cL$ is infinite, $\cL$ admits an accumulation point. So it follows from the second point of Proposition \ref{prop.III-1-permutation-action} that $\cS_G \actson (X,\mu)$ is of type III$_1$. Then, $F$ only depends on the $Y$-variable. Since $G \actson (Y,\eta)$ is ergodic, $F$ is essentially constant. It follows that $G \actson (X,\mu)$ is of stable type III$_1$ whenever $\cL$ is infinite.

Next consider the case where $\cL = \{\lambda_1,\ldots,\lambda_n\}$. Recall that we denote $\gamma(g) = \mu_g(0)$. We can then partition $G$ into $n$ infinite subsets $G = W_1 \sqcup \cdots \sqcup W_n$ such that the function $\zeta : G \recht [\delta,1-\delta]$ defined by $\zeta(g) = \lambda_i$ whenever $g \in W_i$, has the property that $\gamma(g) - \zeta(g) \recht 0$ when $g \recht \infty$. Since $c_g \in \ell^2(G)$, we have in particular that $\lim_{k \recht \infty} c_g(k) = 0$ for every $g \in G$. This implies that each $W_i$ is almost invariant.

First assume that $\gamma - \zeta \not\in \ell^2(G)$. We prove that $G \actson (X,\mu)$ is of stable type III$_1$. Take $i \in \{1,\ldots,n\}$ such that $\sum_{g \in W_i} (\gamma(g) - \lambda_i)^2 = +\infty$. By the first point of Proposition \ref{prop.III-1-permutation-action}, the permutation action
$$\cS_{W_i} \actson \prod_{g \in W_i} (\{0,1\},\mu_g)$$
is of type III$_1$. Since $\cS_{W_i} \subset \cS_G$ and $F$ is $\cS_G$-invariant, it follows in particular that $F$ does not depend on the $\R$-variable. Since $G \actson X \times Y$ is ergodic, we conclude that $F$ is essentially constant, so that $G \actson (X,\mu)$ is of stable type III$_1$.

If $\gamma - \zeta \in \ell^2(G)$, we distinguish the cases $n = 1$ and $n \geq 2$. When $n = 1$, we find that $c$ is a coboundary and $\mu \sim \nu^G$ where $\nu(0) = \lambda_1$, so that $G \actson (X,\mu)$ is of stable type II$_1$.

If $n \geq 2$, we find that $c$ is not a coboundary, but that $c$ is cohomologous to a $1$-cocycle in $\Zai(G,\ell^2(G))$. Replacing $\mu_g(0)$ by $\zeta(g)$, we may assume that $\gamma = \zeta$. Define the subgroup $\Lambda \subset \R$ by \eqref{eq.group-Lambda}. By the second point of Proposition \ref{prop.III-1-permutation-action}, $\exp(\Lambda)$ is a subset of the ratio set of $\cS_G \actson (X,\mu)$. So if $\Lambda \subset \R$ is dense, we get that $\cS_G \actson (X,mu)$ is of type III$_1$ and conclude that $G \actson (X,\mu)$ is of stable type III$_1$.

If $\Lambda = a \Z$ for some $a > 0$, we conclude that $F(x,t+a,y) = F(x,t,y)$ for a.e.\ $(x,t,y) \in X \times \R \times Y$. From the definition of $\Lambda$, it also follows that $\be(\si,x) \in a \Z$ for all $\si \in \cS_G$, $x \in X$. Since $F$ is invariant under the Maharam extension of $\cS_G \actson (X,\mu)$, we conclude that $F$ does not depend on the $X$-variable. Since $\Lambda = a \Z$, it follows from \eqref{eq.essential-algebra} that $\al(g,\tau_s(x)) \in \al(g,x) + a \Z$ for all $g,s \in G$ and a.e.\ $x \in X$. So, modulo $a \Z$, the function $\al(g,x)$ does not depend on $x$. A direct computation then gives that $\al(g,x) \in \Om(g) + a \Z$, where the group homomorphism $\Om$ is defined by \eqref{eq.group-hom-Om}. Altogether we conclude that $L^\infty(X \times \R \times Y)^G$ is precisely given by the functions in $L^\infty(\R/a \Z \times Y)$ that are invariant under the pmp action $G \actson \R/a\Z \times Y$ given by $g \cdot (t,y) = (t + \Om(g),g \cdot y)$.

When $\Om(G) \subset a \Z$, it follows that $L^\infty(X \times \R \times Y)^G = 1 \ot L^\infty(\R/a\Z) \ot 1$ for any choice of ergodic pmp action, so that $G \actson (X,\mu)$ is of stable type III$_\lambda$ with $\lambda = \exp(-a)$.

When $\Om(G) \not\subset a\Z$, we first take $Y$ to be one point. If $a\Z + \Om(G)$ is dense in $\R$, it follows that $G \actson (X,\mu)$ is of type III$_1$. If $a\Z + \Om(G) = b \Z$ for some $b > 0$, it follows that $G \actson (X,\mu)$ is of type III$_\lambda$ with $\lambda = \exp(-b)$. The possible types of $G \actson X \times Y$ that may arise when varying $G \actson (Y,\eta)$ can be found precisely as in \cite[Proposition 7.3]{VW17}.
\end{proof}

\begin{remark}\label{rem.ends-groups}
A countable group $G$ is said to have \emph{more than one end} if there exists an almost invariant subset $W \subset G$ such that both $W$ and $G \setminus W$ are infinite. By Stallings' theorem and its version for groups that are not necessarily finitely generated (see \cite[Theorem IV.6.10]{DD89}), the groups with more than one end are precisely the following groups.
\begin{itemlist}
\item The virtually cyclic groups.
\item The infinite, locally finite groups.
\item The amalgamated free product groups $A *_C B$ with $C$ finite, $C < A$ and $C < B$ proper subgroups and $[A:C]+[B:C] \geq 5$.
\item The HNN extensions $\HNN(A,C,\al)$ with $C$ finite, $C < A$ a proper subgroup and $\al : C \recht A$ an injective group homomorphism.
\end{itemlist}
Recall that $\HNN(A,C,\al)$ is defined as the group generated by $A$ and an element $t$, satisfying the relation $t^{-1} a t = \al(a)$ for all $a \in C$. Also note that if $C$ is finite and $C < A$, $C < B$ both have index $2$, then $A *_C B$ is virtually cyclic. Similarly, when $C$ is finite and $C=A$, the HNN extension $\HNN(A,C,\al)$ is virtually cyclic.
\end{remark}

Whenever $G$ is a countable group, $W \subset G$ is an almost invariant subset and $\lambda \in (0,1)$, the Bernoulli action $G \actson \prod_{g \in G} (\{0,1\},\mu_g)$ with
\begin{equation}\label{eq.candidate-type-III-lambda}
\mu_g(0) = \frac{\lambda}{1+\lambda} \;\;\text{if $g \in W$, and}\;\; \mu_g(0) = \frac{1}{1+\lambda} \;\;\text{if $g \not\in W$,}
\end{equation}
is nonsingular and is a candidate for being of type III$_\lambda$.

However, when $G = \Z$, up to a finite subset, the only almost invariant subsets are $[0,+\infty)$ and $(-\infty,0]$. Then for any $\lambda \in (0,1)$, the above nonsingular Bernoulli action is dissipative. For the same reason, we have to rule out the virtually cyclic groups. And for the ``smallest'' amalgamated free products and HNN extensions, a certain subtlety arises.

Recall that an almost invariant subset $W \subset G$ is said to be nontrivial if both $W$ and $G \setminus W$ are infinite.

\begin{proposition}\label{prop.good-almost-invariant}
For an almost invariant subset $W \subset G$ and a $\kappa > 0$, consider the following two properties.
\begin{equation}\label{eq.properties}
|W \setminus g W| = |g W \setminus W| \quad\text{for all $g \in G$, and}\quad \sum_{g \in G} \exp(-\kappa \, |W \vartriangle g W|) = +\infty \; .
\end{equation}
\begin{enumlist}
\item If $G$ is infinite and locally finite, for every $\kappa > 0$, there exists a nontrivial almost invariant subset $W \subset G$ satisfying \eqref{eq.properties}.
\item If $G = A *_C B$ is an amalgamated free product with $C$ finite, $C < A$ and $C < B$ proper subgroups and $[A:C]+[B:C] \geq 5$, there exists a nontrivial almost invariant subset $W \subset G$ satisfying \eqref{eq.properties} for
    \begin{equation}\label{eq.condition-AFP}
    \kappa \leq \frac{\log([A:C]-1) + \log([B:C]-1)}{2 |C|} \; .
    \end{equation}
\item If $G = \HNN(A,C,\al)$ with $C$ finite, $C < A$ a proper subgroup and $\al : C \recht A$ an injective group homomorphism, there exists a nontrivial almost invariant subset $W \subset G$ satisfying \eqref{eq.properties} for
    \begin{equation}\label{eq.condition-HNN}
    \kappa \leq \frac{\log(2[A:C] - 1)}{2|C|} \; .
    \end{equation}
\end{enumlist}
\end{proposition}

Note that \eqref{eq.condition-AFP} automatically holds when $A$ or $B$ is infinite. Similarly, \eqref{eq.condition-HNN} automatically holds if $A$ is infinite.

Recall from \eqref{eq.group-hom-Om-W} the group homomorphism $\Om_W : G \recht \Z$ associated to an almost invariant subset $W \subset G$. The first condition in \eqref{eq.properties} means that $\Om_W(g) = 0$ for all $g \in G$.

\begin{proof}
1.\ Fix $\kappa > 0$. We construct an almost invariant subset of $G$ using the method in the proof of \cite[Theorem IV.6.10]{DD89}. Write $G$ as the union of a strictly increasing sequence of finite subgroups $G_n \subset G$. After passing to a subsequence, we may assume that
\begin{equation}\label{eq.many-elements}
|G_n \setminus G_{n-1}| \, \exp(- 2\kappa \, |G_{n-1}| ) \geq 1 \; .
\end{equation}
Now let $S \subset \{1,2,3,\ldots\}$ be any infinite subset with infinite complement. Define
$$
W = \bigcup_{n \in S} (G_n \setminus G_{n-1}) \; .
$$
The subset $W \subset G$ is almost invariant. Indeed, given $g \in G$, we can take $m$ large enough such that $g \in G_m$. Then, $g(G_n \setminus G_{n-1}) = G_n \setminus G_{n-1}$ for all $n \geq m+1$.

Fix $n \in S$ and $g \in G_n \setminus G_{n-1}$. We claim that $|W \setminus g^{-1} W| \leq |G_{n-1}|$. Whenever $m \leq n-1$, we have
$$g (G_m \setminus G_{m-1}) \subset g G_m \subset G_n \setminus G_{n-1} \subset W \; .$$
Whenever $m \geq n+1$, we have $g(G_m \setminus G_{m-1}) = G_m \setminus G_{m-1}$. So we already get that
$$W \setminus g^{-1} W \subset G_n \setminus G_{n-1} \; .$$
If $w \in G_n \setminus G_{n-1}$ and $gw \not\in W$, we must have $gw \in G_{n-1}$, because $gw \in G_n$. So, we have proven that $W \setminus g^{-1} W = g^{-1} G_{n-1}$ and the claim follows.

The claim says that $|gW \setminus W| \leq |G_{n-1}|$ for every $n \in S$ and every $g \in G_n \setminus G_{n-1}$. Then also $g^{-1} \in G_n \setminus G_{n-1}$, so that also $|W \setminus gW| \leq |G_{n-1}|$. We thus conclude that $|W \vartriangle g W| \leq 2 |G_{n-1}|$ for all $n \in S$ and all $g \in G_n \setminus G_{n-1}$. It follows that
\begin{align*}
\sum_{g \in W} \exp(-\kappa \; |W \vartriangle g W|) &= \sum_{n \in S} \sum_{g \in G_n \setminus G_{n-1}} \exp(-\kappa \; |W \vartriangle g W|) \\
&\geq \sum_{n \in S} \; |G_n \setminus G_{n-1}| \; \exp(- 2 \kappa \, |G_{n-1}|) = +\infty \; ,
\end{align*}
by using \eqref{eq.many-elements}. Since $\Om_W : G \recht \Z$ is a group homomorphism and $G$ is locally finite, we have that $\Om_W(g) = 0$ for all $g \in G$. So \eqref{eq.properties} holds.

2.\ Let $G = A *_C B$ be an amalgamated free product as in the formulation of the proposition. A word
\begin{equation}\label{eq.word-AFP}
a_0 b_1 a_1 \cdots a_{n-1} b_n a_n \quad\text{with $a_i \in A$ and $b_i \in B$,}
\end{equation}
is said to be \emph{reduced} if $b_i \in B \setminus C$ for all $i \in \{1,\ldots,n\}$ and $a_i \in A \setminus C$ for all $i \in \{1,\ldots,n-1\}$. Reduced words with $n \geq 1$ are never equal to the neutral element in $G$.

We define $W \subset G$ as the set of elements $g \in G$ that admit a reduced expression as in \eqref{eq.word-AFP} with $n \geq 1$ and $a_n = e$. One checks that $W \subset G$ is almost invariant and that the associated $1$-cocycle $c_W(g) = 1_W - 1_{gW}$ satisfies $c_W(a) = 0$ for all $a \in A$ and $c_W(b) = 1_{bC} - 1_C$ for all $b \in B$. In particular, $\Om_W(g) = 0$ for all $g \in G$.

Whenever $g$ is given by a reduced expression as in \eqref{eq.word-AFP}, we get that
$$c_W(g) = (-1_{a_0C} + 1_{a_0b_1 C}) + \cdots + (-1_{a_0 b_1 \cdots a_{n-1}C} + 1_{a_0 b_1 \cdots a_{n-1} b_n C}) \; .$$
It follows that $|W \vartriangle g W| = \|c_W(g)\|_2^2 = 2n |C|$.

Choose representatives $e \in \cA \subset A$ for $A/C$ and $e \in \cB \subset B$ for $B/C$. For a fixed $n \geq 1$, the elements of $G$ admitting a reduced expression as in \eqref{eq.word-AFP} are exactly enumerated by taking $a_0 \in \cA$, $b_i \in \cB \setminus \{e\}$ for all $i=1,\ldots,n$, $a_i \in \cA \setminus \{e\}$ for all $i=1,\ldots,n-1$ and $a_n \in A$. There are thus exactly
$$[A:C] \, ([A:C]-1)^{n-1} \, ([B:C]-1)^n \, |A|$$
such elements. So when $\kappa$ satisfies \eqref{eq.condition-AFP}, then \eqref{eq.properties} holds.

3.\ Let $G = \HNN(A,C,\al)$ be an HNN extension as in the formulation of the proposition. A word
\begin{equation}\label{eq.word-HNN}
a_0 t^{\eps_1} a_1 \cdots a_{n-1} t^{\eps_n} a_n \quad\text{with $a_i \in A$ and $\eps_i \in \{-1,1\}$,}
\end{equation}
is said to be \emph{reduced} if the following two conditions hold: if $1 \leq i \leq n-1$ and $\eps_i = -1$ and $\eps_{i+1}=1$, then
$a_i \in A \setminus C$; if $1 \leq i \leq n-1$ and $\eps_i = 1$ and $\eps_{i+1}=-1$, then $a_i \in A \setminus \al(C)$. Again, reduced words with $n \geq 1$ are never equal to the neutral element in $G$.

We define $W \subset G$ as the set of elements $g \in G$ that admit a reduced expression as in \eqref{eq.word-HNN} with $n \geq 1$ and $a_n = e$. One checks that $W \subset G$ is almost invariant and that the associated $1$-cocycle $c_W(g) = 1_W - 1_{gW}$ satisfies $c_W(a) = 0$ for all $a \in A$ and $C_W(t) = 1_{Ct} - 1_C$. In particular, $\Om_W(g) = 0$ for all $g \in G$.

Denote $C_1 = C$ and $C_{-1} = \al(C)$. Also denote by $g \cdot \xi$ the left translation by $g \in G$ of a function $\xi : G \recht \R$. Whenever $g$ is given by a reduced expression as in \eqref{eq.word-HNN}, we get that
\begin{equation}\label{eq.expression-cocycle}
\begin{split}
c_W(g) = & -1_{a_0 C_{\eps_1}} + a_0 t^{\eps_1} \cdot \bigl(1_{C_{-\eps_1}}-1_{a_1 C_{\eps_2}}\bigr) + \cdots \\ &+ a_0t^{\eps_1}\cdots a_{n-2} t^{\eps_{n-1}} \cdot \bigl(1_{C_{-\eps_{n-1}}}-1_{a_{n-1} C_{\eps_n}}\bigr) + a_0t^{\eps_1}\cdots a_{n-1} t^{\eps_n} \cdot 1_{C_{-\eps_n}} \; .
\end{split}
\end{equation}
Although an element $g \in G$ has several reduced expressions as in \eqref{eq.word-HNN}, the integer $n$ is uniquely determined by the group element. It follows that the decomposition in \eqref{eq.expression-cocycle} is orthogonal, so that $|W \vartriangle gW| = \|c_W(g)\|_2^2 \leq 2n |C|$.

Choose representatives $e \in \cA_{\eps} \subset A$ for $A/C_\eps$. For a fixed $n \geq 1$, the elements of $G$ admitting a reduced expression as in \eqref{eq.word-HNN} are exactly enumerated by taking any sequence $(\eps_1,\ldots,\eps_n)$ with $\eps_i \in \{-1,1\}$, $a_i \in \cA_{\eps_{i+1}}$ for all $i=0,\ldots,n-1$, with $a_i \neq e$ if $1 \leq i \leq n-1$ and $\eps_i \neq \eps_{i+1}$, and $a_n \in A$. To count the number of such elements, consider $(\eps_1,\ldots,\eps_n)$ as a sequence of $k_1$ times $1$, followed by $k_2$ times $-1$, etc., or as a sequence of $k_1$ times $-1$, followed by $k_2$ times $1$, etc. Write $\rho_1 = [A:C]$ and $\rho_2 = [A:C]-1$. So, the number of such elements equals
\begin{align*}
2 \, \sum_{r=1}^n \, & \, \sum_{\substack{k_1,\ldots,k_r \geq 1 \\ k_1+\cdots+k_r = n}} \; \rho_1^{k_1} \, \rho_2 \, \rho_1^{k_2-1} \, \cdots \, \rho_2 \, \rho_1^{k_r-1} \, |A|
\\ &=2 |A| \sum_{r=1}^n \; \rho_1^{n-r+1} \; \rho_2^{r-1} \;\; \#\{(k_1,\ldots,k_r) \mid k_i \geq 1 \; , \;k_1+\cdots+k_r = n\}
\\ &=2 |A| \sum_{r=1}^n \; \rho_1^{n-r+1} \; \rho_2^{r-1} \; \binom{n-1}{r-1} = 2|A|\rho_1 \, (\rho_1 + \rho_2)^{n-1} \; .
\end{align*}
So, when $\kappa$ satisfies \eqref{eq.condition-HNN}, then \eqref{eq.properties} holds.
\end{proof}

We now prove the following more precise version of Theorem \ref{thm.main-general}. By Remark \ref{rem.ends-groups} and Proposition \ref{prop.good-almost-invariant}, Theorem \ref{thm.main-general} is indeed a consequence of the following result.

\begin{theorem}\label{thm.other-version-main-thm}
Let $G$ be a countable group.
\begin{enumlist}
\item If $H^1(G,\ell^2(G)) \neq \{0\}$, then $G$ admits a nonsingular Bernoulli action $G \actson (X,\mu)$ of stable type III$_1$. If $G$ is nonamenable, $G \actson (X,\mu)$ may be chosen nonamenable in the sense of Zimmer.

\item If $G$ admits a nontrivial almost invariant subset $W \subset G$ and if $\kappa > 0$ such that \eqref{eq.properties} holds, then $G$ admits a nonsingular Bernoulli action of stable type III$_\lambda$ for any $\lambda \in (0,1)$ satisfying
    \begin{equation}\label{eq.condition-lambda}
    4 \frac{(1-\lambda)^2}{\lambda} < \kappa \; .
    \end{equation}
\end{enumlist}
\end{theorem}

\begin{proof}
1.\ When $G$ is an infinite amenable group, this was proven in \cite[Corollary 1.4]{BK18}, but can also be deduced as follows. By \cite[Proposition 6.8]{VW17}, we can choose $\mu_g(0) \in [1/3,2/3]$ such that $\lim_{g \recht \infty} \mu_g(0) = 1/2$ and such that $c_g(k) = \mu_k(0) - \mu_{g^{-1} k}(0)$ is a nontrivial $1$-cocycle in $H^1(G,\ell^2(G))$ with $g \mapsto \|c_g\|_2$ growing arbitrarily slowly. By Theorem \ref{thm.type-bernoulli}, we can thus make our choice such that $G \actson (X,\mu)$ is of stable type III$_1$.

If $G$ is nonamenable and $G$ has more than one end, then $G$ has infinitely many ends. We can then partition $G$ into three infinite almost invariant subsets $G = W_1 \sqcup W_2 \sqcup W_3$. Consider the function
$$\vphi_\eps : G \recht \R : \vphi_\eps(g) = \exp\bigl(- \eps \; \bigl( |W_1 \vartriangle g W_1| + |W_2 \vartriangle g W_2| + |W_3 \vartriangle g W_3|\bigr)\bigr) \; .$$
For every $i \in \{1,2,3\}$, $c_{W_i}$ is a $1$-cocycle, so that the function $g \mapsto \|c_{W_i}(g)\|_2^2 = |W_i \vartriangle g W_i|$ is conditionally of negative type. By Schoenberg's theorem (see e.g. \cite[Theorem D.11]{BO08}), for every $\eps > 0$, the functions $g \mapsto \exp\bigl(- \eps \,  |W_i \vartriangle g W_i|\bigr)$ are positive definite, as is their product $\vphi_\eps$.

When $\eps \recht 0$, we have $\vphi_\eps(g) \recht 1$ pointwise. We claim that $\vphi_\eps \not\in \ell^2(G)$ for all $\eps > 0$ small enough. Indeed, otherwise we find a sequence of square summable positive definite functions on $G$ converging to one pointwise. By \cite[Th\'{e}or\`{e}me 17]{God46}, there then also exists a sequence of finitely supported positive definite functions on $G$ converging to one pointwise, contradicting the nonamenability of $G$. So the claim is proven.

For any choice of $0 < \delta_1,\delta_2,\delta_3 < 1/2$, consider the probability measures $\mu_g$ on $\{0,1\}$ given by $\mu_g(0) = 1/2 + \delta_i$ if $g \in W_i$. Since the sets $W_i$ are almost invariant, the Bernoulli action $G \actson (X,\mu) = \prod_{g \in G} (\{0,1\},\mu_g)$ is nonsingular. The associated $1$-cocycle is given by
$$c_g = \sum_{i=1}^3 \delta_i (1_{W_i} - 1_{g W_i}) \quad\text{so that}\quad \|c_g\|_2^2 \leq 3 \sum_{i=1}^3 \delta_i^2 \, |W_i \vartriangle g W_i| \; .$$
Since $\vphi_\eps \not\in \ell^2(G)$ when $\eps > 0$ is small enough, for all $\delta_i$ small enough, the hypotheses of Theorem \ref{thm.type-bernoulli} are satisfied. Combining \cite[Lemma 5.4 and Proposition 5.3]{VW17}, for all $\delta_i$ small enough, the action $G \actson (X,\mu)$ is nonamenable in the sense of Zimmer.

Choosing $\delta_1,\delta_2,\delta_3$ such that
$$\frac{\delta_1}{1-\delta_1} \, \frac{1-\delta_2}{\delta_2} \quad\text{and}\quad \frac{\delta_1}{1-\delta_1} \, \frac{1-\delta_3}{\delta_3}$$
generate a dense subgroup of $\R_*^+$, it follows from Theorem \ref{thm.type-bernoulli} that $G \actson (X,\mu)$ is of stable type III$_1$.

Finally, if $G$ is nonamenable and $G$ has one end, all $1$-cocycles in $\Zai(G,\ell^2(G))$ are coboundary. Since $H^1(G,\ell^2(G)) \neq \{0\}$, after some rescaling, we find a nonconstant function $\gamma : G \recht [0,1]$ such that $c_g = \gamma - g \cdot \gamma$ belongs to $\ell^2(G)$ for every $g \in G$. Given $0 < \eps < 1/2$, define the probability measures $\mu_g(0) = 1/2 + \eps \gamma(g)$. Using Schoenberg's theorem as above, it follows that for $\eps > 0$ small enough, the hypotheses of Theorem \ref{thm.type-bernoulli} are satisfied and the Bernoulli action $G \actson (X,\mu)$ is nonamenable in the sense of Zimmer. By Theorem \ref{thm.type-bernoulli}, the action $G \actson (X,\mu)$ is of stable type III$_1$.

2.\ Define the probability measures $\mu_g$ on $\{0,1\}$ given by \eqref{eq.candidate-type-III-lambda}. If \eqref{eq.condition-lambda} holds, then \eqref{eq.properties} says that the hypotheses of Theorem \ref{thm.type-bernoulli} hold. So, by Theorem \ref{thm.type-bernoulli}, $G \actson (X,\mu)$ is of stable type III$_\lambda$.
\end{proof}

\end{document}